\documentclass[reqno]{pspum-l}
\usepackage{amsmath,amsthm,amscd,amssymb}
\usepackage{latexsym}
\usepackage{graphicx}
\newcommand{\bbC}{{\mathbb{C}}}
\newcommand{\bbD}{{\mathbb{D}}}

\newcommand{\bbR}{{\mathbb{R}}}

\newcommand{\bbZ}{{\mathbb{Z}}}

\newcommand{\calE}{{\mathcal{E}}}

\newcommand{\calJ}{{\mathcal J}}

\newcommand{\fre}{{\frak{e}}}

\newcommand{\bdone}{{\boldsymbol{1}}}
\newcommand{\bddot}{{\boldsymbol{\cdot}}}

\newcommand{\dott}{\,\cdot\,}

\newcommand{\lb}{\label}
\newcommand{\f}{\frac}

\newcommand{\ol}{\overline}
\newcommand{\ti}{\tilde  }
\newcommand{\wti}{\widetilde  }

\newcommand{\Lt}{\text{\rm{L}}}
\newcommand{\Rt}{\text{\rm{R}}}

\newcommand{\tr}{\text{\rm{Tr}}}

\newcommand{\dist}{\text{\rm{dist}}}

\newcommand{\ran}{\text{\rm{ran}}}

\newcommand{\ess}{\text{\rm{ess}}}
\newcommand{\ac}{\text{\rm{ac}}}
\newcommand{\singc}{\text{\rm{sc}}}
\newcommand{\s}{\text{\rm{s}}}

\newcommand{\pp}{\text{\rm{pp}}}
\newcommand{\supp}{\text{\rm{supp}}}
\newcommand{\intt}{\text{\rm{int}}}

\newcommand{\bi}{\bibitem}

\newcommand{\beq}{\begin{equation}}
\newcommand{\eeq}{\end{equation}}
\newcommand{\ba}{\begin{align}}
\newcommand{\ea}{\end{align}}
\newcommand{\veps}{\varepsilon}

\let\det=\undefined\DeclareMathOperator{\det}{det}

\newcommand{\ang}[1]{\langle\! \langle#1\rangle\! \rangle}

\newcounter{smalllist}
\newenvironment{SL}{\begin{list}{{\rm\roman{smalllist})}}{%
\setlength{\topsep}{0mm}\setlength{\parsep}{0mm}\setlength{\itemsep}{0mm}%
\setlength{\labelwidth}{2em}\setlength{\leftmargin}{2em}\usecounter{smalllist}%
}}{\end{list}}

\DeclareMathOperator{\Ima}{Im}
\DeclareMathOperator{\diam}{diam}

\allowdisplaybreaks
\numberwithin{equation}{section}

\newtheorem{theorem}{Theorem}[section]

\newtheorem*{p2.1}{Proposition 2.1}
\newtheorem{proposition}[theorem]{Proposition}
\newtheorem{lemma}[theorem]{Lemma}
\newtheorem{corollary}[theorem]{Corollary}
\theoremstyle{definition}
\newtheorem{example}[theorem]{Example}
\newtheorem{conjecture}[theorem]{Conjecture}

\theoremstyle{remark}
\newtheorem*{remark}{Remark}
\newtheorem*{remarks}{Remarks}
\newtheorem*{definition}{Definition}

\newcommand{\abs}[1]{\lvert#1\rvert}
\newcommand{\Norm}[1]{\lVert#1\rVert}
\newcommand{\jap}[1]{\langle #1 \rangle}

\begin{document}
\title{The Christoffel--Darboux Kernel}

\author[B.~Simon]{Barry Simon$^*$}

\address{Mathematics 253--37, California Institute of Technology, Pasadena, CA 91125,
U.S.A.}

\email{bsimon@caltech.edu}

\urladdr{http://www.math.caltech.edu/people/simon.html}

\thanks{This work was supported in part by NSF grant DMS-0652919 and
U.S.--Israel Binational Science Foundation (BSF) Grant No.\ 2002068}


\begin{abstract} A review of the uses of the CD kernel in the spectral theory of
orthogonal polynomials, concentrating on recent results.
\end{abstract}

\maketitle

\setcounter{tocdepth}{1}
\tableofcontents

\section{Introduction} \lb{s1}

This article reviews a particular tool of the spectral theory of orthogonal polynomials. Let $\mu$
be a measure on $\bbC$ with finite moments, that is,
\begin{equation}  \lb{1.1}
\int \abs{z}^n\, d\mu(z) <\infty
\end{equation}
for all $n=0,1,2,\dots$ and which is nontrivial in the sense that it is not supported on a finite set
of points. Thus, $\{z^n\}_{n=0}^\infty$ are independent in $L^2 (\bbC,d\mu)$, so by Gram--Schmidt,
one can define monic orthogonal polynomials, $X_n(z;d\mu)$, and orthonormal polynomials, $x_n =
X_n/\Norm{X_n}_{L^2}$. Thus,
\begin{gather}
\int \bar z^j X_n(z;d\mu) \, d\mu(z) = 0 \qquad j=0, \dots, n-1 \lb{1.2} \\
X_n(z) = z^n + \text{lower order} \lb{1.3} \\
\int \ol{x_n(z)}\, x_m(z)\, d\mu =\delta_{nm} \lb{1.4}
\end{gather}

We will often be interested in the special cases where $\mu$ is supported on $\bbR$ (especially
with support compact), in which case we use $P_n,p_n$ rather than $X_n,x_n$, and where $\mu$ is
supported on $\partial\bbD$ ($\bbD=\{z\mid\abs{z}<1\}$), in which case we use $\Phi_n,\varphi_n$.
We call these OPRL and OPUC (for ``real line'' and ``unit circle'').

OPRL and OPUC are spectral theoretic because there are Jacobi parameters $\{a_n,b_n\}_{n=1}^\infty$
and Verblunsky coefficients $\{\alpha_n\}_{n=0}^\infty$ with recursion relations ($p_{-1}=0$; $p_0=
\Phi_0=1$):
\begin{align}
zp_n(z) &= a_{n+1} p_{n+1}(z) + b_{n+1} p_n(z) + a_n p_{n-1}(z) \lb{1.5} \\
\Phi_{n+1}(z) &= z\Phi_n(z) - \bar\alpha_n \Phi_n^*(z) \lb{1.6} \\
\Phi_n^*(z) &= z^n \,\ol{\Phi_n (1/\bar z)} \lb{1.7}
\end{align}

We will sometimes need the monic OPRL and normalized OPUC recursion relations:
\begin{align}
& zP_n(z) = P_{n+1}(z) + b_{n+1} P_n(z) + a_n^2 P_{n-1}(z) \lb{1.7a} \\
& z\varphi_n(z) = \rho_n \varphi_{n+1}(z) + \bar\alpha_n \varphi_n^*(z) \lb{1.7b} \\
& \rho_n \equiv (1-\abs{\alpha_n}^2)^{1/2} \lb{1.7c}
\end{align}
Of course, the use of $\rho_n$ implies $\abs{\alpha_n}<1$ and all sets of $\{\alpha_n\}_{n=0}^\infty$
obeying this occur. Similarly, $b_n\in\bbR$, $a_n\in (0,\infty)$ and all such sets occur. In the OPUC
case, $\{\alpha_n\}_{n=0}^\infty$ determine $d\mu$, while in the OPRL case, they do if $\sup(\abs{a_n}
+ \abs{b_n}) <\infty$, and may or may not in the unbounded case. For basics of OPRL, see
\cite{Szb,Chi,FrB,Rice}; and for basics of OPUC, see \cite{Szb,GBk,FrB,OPUC1,OPUC2,1Ft}.

We will use $\kappa_n$ (or $\kappa_n(d\mu)$) for the leading coefficient of $x_n$, $p_n$, or
$\varphi_n$, so
\begin{equation} \lb{1.10a}
\kappa_n = \Norm{X_n}_{L^2 (d\mu)}^{-1}
\end{equation}

The {\it Christoffel--Darboux kernel} (named after \cite{Chris,Dar}) is defined by
\begin{equation}  \lb{1.8}
K_n(z,\zeta) = \sum_{j=0}^n \ol{x_j(z)}\, x_j(\zeta)
\end{equation}
We sometimes use $K_n(z,\zeta;\mu)$ if we need to make the measure explicit. Note that if $c>0$,
\begin{equation}  \lb{1.8a}
K_n(z,\zeta;c\mu) = c^{-1} K_n(z,\zeta;\mu)
\end{equation}
since $x_n(z; cd\mu) = c^{-1/2} x_n (z;d\mu)$.

By the Schwarz inequality, we have
\begin{equation}\lb{1.13a}
\abs{K_n(z,\zeta)}^2 \leq K_n(z,z) K_n(\zeta,\zeta)
\end{equation}

There are three variations of convention. Some only sum to $n-1$; this is the more common convention
but \eqref{1.8} is used by Szeg\H{o} \cite{Szb}, Atkinson \cite{Atk}, and in \cite{OPUC1,OPUC2}.
As we will note shortly, it would be more natural to put the complex conjugate on $x_n(\zeta)$, not
$x_n(z)$---and a very few authors do that. For OPRL with $z,\zeta$ real, the complex conjugate is
irrelevant---and some authors leave it off even for complex $z$ and $\zeta$.

As a tool in spectral analysis, convergence of OP expansions, and other aspects of analysis, the use
of the CD kernel has been especially exploited by Freud and Nevai, and summarized in Nevai's paper
on the subject \cite{NevFr}. A series of recent papers by Lubinsky (of which \cite{Lub,Lub-jdam} are
most spectacular)  has caused heightened interest in the subject and motivated me to write this
comprehensive review.

Without realizing they were dealing with OPRL CD kernels, these objects have been used extensively in
the spectral theory community, especially the diagonal kernel
\begin{equation}  \lb{1.8b}
K_n(x,x) = \sum_{j=0}^n\, \abs{p_j(x)}^2
\end{equation}
Continuum analogs of the ratios of this object for the first and second kind polynomials appeared
in the work of Gilbert--Pearson \cite{GP} and the discrete analog in Khan--Pearson \cite{KPe} and
then in Jitomirskaya--Last \cite{JL1}. Last--Simon \cite{LastS} studied $\f{1}{n} K_n(x,x)$ as
$n\to\infty$. Variation of parameters played a role in all these works and it exploits what is
essentially mixed CD kernels (see Section~\ref{s8}).

One of our goals here is to emphasize the operator theoretic point of view, which is often
underemphasized in the OP literature. In particular, in describing $\mu$, we think of the
operator $M_z$ on $L^2 (\bbC,d\mu)$ of multiplication by $z$:
\begin{equation}  \lb{1.9}
(M_z f)(z) = zf(z)
\end{equation}
If $\supp(d\mu)$ is compact, $M_z$ is a bounded operator defined on all of $L^2 (\bbC,d\mu)$. If
it is not compact, there are issues of domain, essential selfadjointness, etc.\ that will not
concern us here, except to note that in the OPRL case, they are connected to uniqueness of the
solution of the moment problem (see \cite{S270}). With this in mind, we use $\sigma(d\mu)$
for the spectrum of $M_z$, that is, the support of $d\mu$, and $\sigma_\ess(d\mu)$ for the
essential spectrum. When dealing with OPRL of compact support (where $M_z$ is bounded selfadjoint)
or OPUC (where $M_z$ is unitary), we will sometimes use $\sigma_\ac(d\mu)$, $\sigma_\singc (d\mu)$,
$\sigma_\pp (d\mu)$ for the spectral theory components. (We will discuss $\sigma_\ess(d\mu)$ only
in the OPUC/OPRL case where it is unambiguous, but for general operators, there are multiple
definitions; see \cite{EdEv}.)

The basis of operator theoretic approaches to the study of the CD kernel depends on its interpretation
as the integral kernel of a projection. In $L^2 (\bbC,d\mu)$, the set of polynomials of degree at most
$n$ is an $n+1$-dimensional space. We will use $\pi_n$ for the operator of orthogonal projection onto
this space. Note that
\begin{equation}  \lb{1.10}
(\pi_n f)(\zeta) = \int K_n(z,\zeta) f(z)\, d\mu(z)
\end{equation}
The order of $z$ and $\zeta$ is the opposite of the usual for integral kernels and why we mentioned
that putting complex conjugation on $x_n(\zeta)$ might be more natural in \eqref{1.8}.

In particular,
\begin{equation}  \lb{1.11}
\deg(f)\leq n \Rightarrow f(\zeta) = \int K_n(z,\zeta) f(z)\, d\mu(z)
\end{equation}
In particular, since $K_n$ is a polynomial in $\zeta$ of degree $n$, we have
\begin{equation}  \lb{1.12}
K_n(z,w) =\int K_n(z,\zeta) K_n(\zeta,w)\, d\mu(\zeta)
\end{equation}
often called the reproducing property.

One major theme here is the frequent use of operator theory, for example, proving the CD formula
as a statement about operator commutators. Another theme, motivated by Lubinsky \cite{Lub,Lub-jdam},
is the study of asymptotics of $\f{1}{n} K_n(x,y)$ on diagonal $(x=y$) and slightly off diagonal
($(x-y)=O(\f{1}{n})$).

Sections~\ref{s2}, \ref{s3}, and \ref{s6} discuss very basic formulae, and Sections~\ref{s4} and
\ref{s7} simple applications. Sections~\ref{s5} and \ref{s8} discuss extensions of the context
of CD kernels. Section~\ref{s9} starts a long riff on the use of the Christoffel variational principle
which runs through Section~\ref{s23-new}. Section~\ref{s23} is a final simple application.

\medskip
Vladimir Maz'ya has been an important figure in the spectral analysis of partial differential operators.
While difference equations are somewhat further from his opus, they are related. It is a pleasure
to dedicate this article with best wishes on his 70th birthday.

\smallskip
I would like to thank J.~Christiansen for producing Figure~1 (in Section~\ref{s7}) in Maple,
and C.~Berg, F.~Gesztesy, L.~Golinskii, D.~Lubinsky, F.~Marcell\'an, E.~Saff, and V.~Totik
for useful discussions.

\section{The ABC Theorem} \lb{s2}

We begin with a result that is an aside which we include because it deserves to be better known.
It was rediscovered and popularized by Berg \cite{Bergppt}, who found it earliest in a 1939
paper of Collar \cite{Coll}, who attributes it to his teacher, Aitken---so we dub it the ABC theorem.
Given that it is essentially a result about Gram--Schmidt, as we shall see, it is likely it really
goes back to the nineteenth century. For applications of this theorem, see \cite{Boro,Joh}.

$K_n$ is a polynomial of degree $n$ in $\bar z$ and $\zeta$, so we can define an $(n+1)\times (n+1)$
square matrix, $k^{(n)}$, with entries $k_{jm}^{(n)}$, $0\leq j,m\leq n$, by
\begin{equation}  \lb{2.1}
K_n(z,\zeta) = \sum_{j,m=0}^n k_{jm}^{(n)} \bar z^m \zeta^j
\end{equation}

One also has the moment matrix
\begin{equation}  \lb{2.2}
m_{jk}^{(n)} =\jap{z^j,z^k} = \int \bar z^j z^k\, d\mu(z)
\end{equation}
$0 \leq j,k\leq n$. For OPRL, this is a function of $j+k$, so $m^{(n)}$ is a Hankel matrix. For OPUC,
this is a function of $j-k$, so $m^{(n)}$ is a Toeplitz matrix.

\begin{theorem}[ABC Theorem]\lb{T2.1}
\begin{equation}  \lb{2.2a}
(m^{(n)})^{-1} = k^{(n)}
\end{equation}
\end{theorem}

\begin{proof} By \eqref{1.11} for $\ell=0, \dots, n$,
\begin{equation}  \lb{2.3}
\int K_n(z,\zeta) z^\ell\, d\mu(z) = \zeta^\ell
\end{equation}
Plugging \eqref{2.1} in for $K$, using \eqref{2.2} to do the integrals leads to
\begin{equation}  \lb{2.4}
\sum_{j,q=0}^n k_{jq}^{(n)} m_{q\ell}^{(n)} \zeta^j = \zeta^\ell
\end{equation}
which says that
\begin{equation}  \lb{2.5}
\sum_q k_{jq}^{(n)} m_{q\ell}^{(n)} =\delta_{j\ell}
\end{equation}
which is \eqref{2.2a}.
\end{proof}

Here is a second way to see this result in a more general context: Write
\begin{equation}  \lb{2.6}
x_j(z) =\sum_{k=0}^j a_{jk} z^k
\end{equation}
so we can define an $(n+1)\times (n+1)$ triangular matrix $a^{(n)}$ by
\begin{equation}  \lb{2.7}
a_{jk}^{(n)} =a_{jk}
\end{equation}
Then (the Cholesky factorization of $k$)
\begin{equation}  \lb{2.7x}
k^{(n)} = a^{(n)} (a^{(n)})^*
\end{equation}
with ${}^*$ Hermitean adjoint. The condition
\begin{equation}  \lb{2.8}
\jap{x_j,x_\ell} =\delta_{j\ell}
\end{equation}
says that
\begin{equation}  \lb{2.9}
(a^{(n)})^* m^{(n)} (a^{(n)}) =\bdone
\end{equation}
the identity matrix. Multiplying by $(a^{(n)})^*$ on the right and $[(a^{(n)})^*]^{-1}$ on
the left yields \eqref{2.2a}. This has a clear extension to a general Gram--Schmidt setting.

\section{The Christoffel--Darboux Formula} \lb{s3}

The Christoffel--Darboux formula for OPRL says that
\begin{equation}  \lb{3.1}
K_n(z,\zeta) = a_{n+1} \biggl( \f{\ol{p_{n+1}(z)}\, p_n(\zeta) - \ol{p_n(z)}\, p_{n+1} (\zeta)}
{\bar z-\zeta}\biggr)
\end{equation}
and for OPUC that
\begin{equation}  \lb{3.2}
K_n(z,\zeta) = \f{\ol{\varphi_{n+1}^*(z)}\, \varphi_{n+1}^*(\zeta) -\ol{\varphi_{n+1}(z)}\,
\varphi_{n+1}(\zeta)}{1-\bar z\zeta}
\end{equation}
The conventional wisdom is that there is no CD formula for general OPs, but we will see, in a sense,
that is only half true. The usual proofs are inductive. Our proofs here will be direct operator theoretic
calculations.

We focus first on \eqref{3.1}. From the operator point of view, the key is to note that, by \eqref{1.10},
\begin{equation}  \lb{3.3}
\jap{g,[M_z,\pi_n]f} = \int \ol{g(\zeta)}\, (\bar\zeta-z) K_n(z,\zeta)f(z)\, d\mu(\zeta) d\mu(z)
\end{equation}
where $[A,B]=AB-BA$. For OPRL, in \eqref{3.3}, $\zeta$ and $z$ are real, so \eqref{3.1} for $z,\zeta
\in\sigma(d\mu)$ is equivalent to
\begin{equation}  \lb{3.4}
[M_z,\pi_n] = a_{n+1} [\jap{p_n, \dott} p_{n+1} - \jap{p_{n+1}, \dott} p_n]
\end{equation}
While \eqref{3.4} only proves \eqref{3.1} for such $\bar z,\zeta$ by the fact that both sides are
polynomials in $z$ and $\zeta$, it is actually equivalent. Here is the general result:

\begin{theorem}[General Half CD Formula]\lb{T3.1} Let $\mu$ be a measure on $\bbC$ with finite
moments. Then:
\begin{gather}
(1-\pi_n) [M_z,\pi_n](1-\pi_n) = 0 \lb{3.5} \\
\pi_n [M_z,\pi_n]\pi_n = 0 \lb{3.6} \\
(1-\pi_n) [M_z,\pi_n]\pi_n = \f{\Norm{X_{n+1}}}{\Norm{X_n}}\, \jap{x_n,\dott}x_{n+1} \lb{3.7}
\end{gather}
\end{theorem}

\begin{remark} If $\mu$ has compact support, these are formulae involving bounded operators
on $L^2 (\bbC,d\mu)$. If not, regard $\pi_n$ and $M_z$ as maps of polynomials to polynomials.
\end{remark}

\begin{proof} \eqref{3.5} follows from expanding $[M_z,\pi_n]$ and using
\begin{equation}  \lb{3.8}
\pi_n (1-\pi_n) = (1-\pi_n) \pi_n =0
\end{equation}
If we note that $[M_z,\pi_n]=-[M_z, (1-\pi_n)]$, \eqref{3.6} similarly follows from \eqref{3.8}.
By \eqref{3.8} again,
\begin{equation}  \lb{3.9}
(1-\pi_n) [M_z,\pi_n]\pi_n = (1-\pi_n) M_z \pi_n
\end{equation}
On $\ran (\pi_{n-1})$, $\pi_n$ is the identity, and multiplication by $z$ leaves one in $\pi_n$,
that is,
\begin{equation}  \lb{3.10}
(1-\pi_n) M_z \pi_n \restriction \ran (\pi_{n-1}) =0
\end{equation}
On the other hand, for the monic OPs,
\begin{equation}  \lb{3.11}
(1-\pi_n) M_z \pi_n X_n = X_{n+1}
\end{equation}
since $M_z \pi_n X_n = z^{n+1} +$ lower order and $(1-\pi_n)$ takes any such polynomial to $X_{n+1}$.
Since
\[
\f{\Norm{X_{n+1}}}{\Norm{X_n}}\, \jap{x_n,X_n} x_{n+1} = X_{n+1}
\]
we see \eqref{3.4} holds on $\ran(1-\pi_n)+\ran(\pi_{n-1}) + [X_n]$, and so on all of $L^2$.
\end{proof}

From this point of view, we can understand what is missing for a CD formula for general OP. The
missing piece is
\begin{equation}  \lb{3.8x}
\pi_n [M_z,\pi_n] (1-\pi_n) = ((1-\pi_n)M_z^* \pi_n)^*
\end{equation}
The operator on the left of \eqref{3.7} is proven to be rank one, but $(1-\pi_n) M_z^* \pi_n$
is, in general, rank $n$. For $\varphi\in \ker[(1-\pi_n) M_z^* \pi_n]\cap\ran(\pi_n)$ means that
$\varphi$ is a polynomial of degree $n$ and so is $\bar z\varphi$, at least for a.e.\ $z$ with
respect to $\mu$. Two cases where many $\bar z\varphi$ are polynomials of degree $n$---indeed,
so many that $(1-\pi_n)M_z^* \pi_n$ is also rank one---are for OPRL where $\bar z\varphi=z\varphi$
(a.e.\ $z\in \sigma(d\mu)$) and OPUC where $\bar z\varphi =z^{-1} \varphi$ (a.e.\ $z\in\sigma(d\mu)$).
In the first case, $\bar z\varphi\in\ran (\pi_n)$ if $\deg(\varphi)\leq n-1$, and in the second case,
if $\varphi(0) =0$.

Thus, only for these two cases do we expect a simple formula for $[M_z,\pi]$.

\begin{theorem}[CD Formula for OPRL]\lb{T3.2} For OPRL, we have
\begin{equation}  \lb{3.9x}
[M_z,\pi_n] = a_{n+1} [\jap{p_n,\dott} p_{n+1} - \jap{p_{n+1}, \dott} p_n]
\end{equation}
and \eqref{3.1} holds for $\bar z\neq \zeta$.
\end{theorem}

\begin{proof} Inductively, one has that $p_n(x) =(a_1\dots a_n)^{-1} x^n + \dots$, so
\begin{equation}  \lb{3.10x}
\Norm{P_n} = a_1 \dots a_n \mu(\bbR)^{1/2}
\end{equation}
and thus,
\begin{equation}  \lb{3.10y}
\f{\Norm{P_{n+1}}}{\Norm{P_n}} = a_{n+1}
\end{equation}

Moreover, since $M_z^*=M_z$ for OPRL and $[A,B]^*=-[A^*,B^*]$, we get from \eqref{3.8x} that
\begin{equation}  \lb{3.11x}
\pi_n [M_z,\pi_n] (1-\pi_n) = -a_{n+1} \jap{p_{n+1}, \dott}p_n
\end{equation}
\eqref{3.5}--\eqref{3.7}, \eqref{3.10x}, and \eqref{3.11x} imply \eqref{3.9x} which, as noted,
implies \eqref{3.1}.
\end{proof}

For OPUC, the natural object is (note $M_z M_z^* =M_z^* M_z =1$)
\begin{equation}  \lb{3.12}
B_n = \pi_n -M_z \pi_n M_z^* =-[M_z, \pi_n] M_z^*
\end{equation}

\begin{theorem}[CD Formula for OPUC]\lb{T3.3} For OPUC, we have
\begin{equation}  \lb{3.13}
\pi_n -M_z \pi_n M_z^* =\jap{\varphi_{n+1}^*, \dott} \varphi_{n+1}^* -
\jap{\varphi_{n+1}, \dott} \varphi_{n+1}
\end{equation}
and \eqref{3.2} holds.
\end{theorem}

\begin{proof} $B_n$ is selfadjoint so $\ran(B_n) = \ker(B_n)^\perp$. Clearly, $\ran(B_n)\subset
\ran(\pi_n) + M_z [\ran(\pi_n)]=\ran(\pi_{n+1})$ and $B_n z^\ell =0$ for $\ell=1, \dots, n$, so
$\ran(B_n) =\{z,z^2, \dots, z^n\}^\perp\cap\ran(\pi_{n+1})$ is spanned by $\varphi_{n+1}$ and
$\varphi_{n+1}^*$. Thus, both $B_n$ and the right side of \eqref{3.13} are rank two selfadjoint
operators with the same range and both have trace $0$. Thus, it suffices to find a single vector
$\eta$ in the span of $\varphi_{n+1}$ and $\varphi_{n+1}^*$ with $B_n\eta =(\text{RHS of \eqref{3.13}})
\eta$, since a rank at most one selfadjoint operator with zero trace is zero!

We will take $\eta=z\varphi_n$, which lies in the span since, by \eqref{1.7b} and its ${}^*\,$,
\begin{equation}  \lb{3.14}
\rho_n \varphi_{n+1} =z\varphi_n -\bar\alpha_n \varphi_n^*  \qquad
\rho_n \varphi_{n+1}^* =\varphi_n^* -\alpha_n z \varphi_n
\end{equation}

By \eqref{3.11x}, \eqref{3.12}, and
\begin{equation}  \lb{3.15}
\Norm{\Phi_n} = \rho_0 \dots \rho_{n-1} \mu(\partial\bbD)
\end{equation}
we have that
\begin{align}
B_n (z\varphi_n) &= [\pi_n, M_z] \varphi_n \notag \\
&= - (1-\pi_n) M_z \pi_n \varphi_n \notag \\
&= -\rho_n \varphi_{n+1} \lb{3.16}
\end{align}

On the other hand, $\varphi_{n+1}^* \perp\{z,\dots, z^{n+1}\}$, so
\[
\jap{\varphi_{n+1}^*, z\varphi_n} = 0
\]
and, by \eqref{3.14},
\begin{align*}
\jap{\varphi_{n+1}, z\varphi_n} &= \rho_n \jap{\varphi_{n+1},\varphi_{n+1}} +
\bar\alpha_n \jap{\varphi_{n+1}, \varphi_n^*} \\
&= \rho_n
\end{align*}
so
\begin{equation}  \lb{3.17}
[\text{LHS of \eqref{3.13}}] z\varphi_n = -\rho_n \varphi_{n+1}
\end{equation}
\end{proof}

Note that \eqref{3.2} implies
\begin{align*}
K_{n+1}(z,\zeta) &= K_n (z,\zeta) + \ol{\varphi_{n+1}(z)}\, \varphi_{n+1}(\zeta)
\biggl( \f{1-\bar z\zeta}{1-\bar z\zeta}\biggr) \\
&= \f{\ol{\varphi_{n+1}^*(z)}\, \varphi_{n +1}^*(\zeta) - \bar z \zeta\, \ol{\varphi_{n+1}(z)}\,
\varphi_{n+1}(\zeta)}{1-\bar z\zeta}
\end{align*}
so changing index, we get the ``other form'' of the CD formula for OPUC,
\begin{equation}  \lb{3.18}
K_n(z,\zeta) = \f{\ol{\varphi_n(z)^*}\, \varphi_n^*(\zeta) - \ol{z\varphi_n(z)}\, \zeta\varphi_n(\zeta)}
{1-\bar z \varphi}
\end{equation}

We also note that Szeg\H{o} \cite{Szb} derived the recursion relation from the CD formula, so the lack
of a CD formula for general OPs explains the lack of a recursion relation in general.

\section{Zeros of OPRL: Basics Via CD} \lb{s4}

In this section, we will use the CD formula to derive the basic facts about the zeros of OPRL. In the vast
literature on OPRL, we suspect this is known but we don't know where. We were motivated to look for this
by a paper of Wong \cite{Wong}, who derived the basics for zeros of POPUC (paraorthogonal polynomials on
the unit circle) using the CD formula (for other approaches to zeros of POPUCs, see \cite{CMV02,S308}).
We begin with the CD formula on diagonal:

\begin{theorem}\lb{T4.1} For OPRL and $x$ real,
\begin{equation}  \lb{4.1}
\sum_{j=0}^n\, \abs{p_j(x)}^2 = a_{n+1} [p'_{n+1}(x) p_n(x) - p'_n(x) p_{n+1}(x)]
\end{equation}
\end{theorem}

\begin{proof} In \eqref{3.1} with $z=x$, $\zeta =y$ both real, subtract $p_{n+1}(y) p_n(y)$ from both
products on the left and take the limit as $y\to x$.
\end{proof}

\begin{corollary}\lb{C4.2} If $p_n(x_0)=0$ for $x_0$ real, then
\begin{equation}  \lb{4.2}
p_{n+1}(x_0) p'_n(x_0) < 0
\end{equation}
\end{corollary}

\begin{proof} The left-hand side of \eqref{4.1} is strictly positive since $p_0(x)=1$.
\end{proof}

\begin{theorem}\lb{T4.3} All the zeros of $p_n(x)$ are real and simple and the zeros of $p_{n+1}$
strictly interlace those of $p_n$. That is, between any successive zeros of $p_{n+1}$ lies exactly
one zero of $p_n$ and it is strictly between, and $p_{n+1}$ has one zero between each successive
zero of $p_n$ and it has one zero above the top zero of $p_n$ and one below the bottom zero of $p_n$.
\end{theorem}

\begin{proof} By \eqref{4.2}, $p_n(x_0) =0\Rightarrow p'_n(x_0)\neq 0$, so zeros are simple, which
then implies that the sign of $p'_n$ changes between its successive zeros. By \eqref{4.2}, the sign
of $p_{n+1}$ thus changes between zeros of $p_n$, so $p_{n+1}$ has an odd number of zeros between
zeros of $p_n$.

$p_1$ is a real polynomial, so it has one real zero. For $x$ large, $p_n(x) >0$ since the leading
coefficient is positive. Thus, $p'_n(x_0) >0$ at the top zero. From \eqref{4.2}, $p_{n+1}(x_0)<0$
and thus, since $p_{n+1}(x)>0$ for $x$ large, $p_{n+1}$ has a zero above the top zero of $p_n$.
Similarly, it has a zero below the bottom zero.

We thus see inductively, starting with $p_1$, that $p_n$ has $n$ real zeros and they interlace
those of $p_{n-1}$.
\end{proof}

We note that Ambroladze \cite{Amb} and then Denisov--Simon \cite{DenS} used properties of the
CD kernel to prove results about zeros (see Wong \cite{Wong} for the OPUC analog); the latter 
paper includes:

\begin{theorem}\lb{T4.4} Suppose $a_\infty =\sup_n a_n <\infty$ and $x_0\in\bbR$ has $d=\dist(x_0,
\sigma(d\mu)) >0$. Let $\delta = d^2/(d+\sqrt{2}\, a_\infty)$. Then at least one of $p_n$ and $p_{n-1}$
has no zeros in $(x_0-\delta, x_0+\delta)$.
\end{theorem}

They also have results about zeros near isolated points of $\sigma(d\mu)$.

\section{The CD Kernel and Formula for MOPs} \lb{s5}

Given an $\ell\times\ell$ matrix-valued  measure, there is a rich structure of matrix OPs (MOPRL
and MOPUC). A huge literature is surveyed and extended in \cite{DPS}. In particular, the CD kernel and
CD formula for MORL are discussed in Sections~2.6 and 2.7, and for MOPUC in Section~3.4.

There are two ``inner products,'' maps from $L^2$ matrix-valued functions to matrices,
$\ang{\dott,\dott}_\Rt$ and $\ang{\dott,\dott}_\Lt$. The $\Rt$ for right comes from the form
of scalar homogeneity, for example,
\begin{equation}  \lb{5.1}
\ang{f,gA}_\Rt = \ang{f,g}_\Rt A
\end{equation}
but $\ang{f,Ag}_\Rt$ is not related to $\ang{f,g}_\Rt$.

There are two normalized OPs, $p_j^\Rt(x)$ and $p_j^\Lt(x)$, orthonormal in $\ang{\dott,\dott}_\Rt$
and $\ang{\dott,\dott}_\Lt$, respectively, but a single CD kernel (for $z,w$ real and ${}^\dagger$ is
matrix adjoint),
\begin{align}
K_n(z,w) &= \sum_{k=0}^n p_k^\Rt(z) p_k^\Rt(w)^\dagger  \lb{5.2} \\
&=\sum_{k=0}^n p_k^\Lt (z)^\dagger p_k^\Lt(w)  \lb{5.3}
\end{align}
One has that
\begin{equation}  \lb{5.4}
\ang{K_n(\dott,z), f(\dott)}_\Rt = (\pi_n f)(z)
\end{equation}
where $\pi_n$ is the projection in the $\tr(\ang{\dott,\dott}_\Rt)$ inner product to polynomials of
degree $n$.

In \cite{DPS}, the CD formula is proven using Wronskian calculations. We note here that the
commutator proof we give in Section~\ref{s3} extends to this matrix case.

Within the Toeplitz matrix literature community, a result equivalent to the CD formula is
called the Gohberg--Semencul formula; see \cite{BASha,Fuh,GoKr,GoSe,KVM,Tren64,Tren65}.

\section{Gaussian Quadrature} \lb{s6}

Orthogonal polynomials allow one to approximate integrals over a measure $d\mu$ on $\bbR$
by certain discrete measures. The weights in these discrete measures depend on $K_n(x,x)$.
Here we present an operator theoretic way of understanding this.

Fix $n$ and, for $b\in\bbR$, let $J_{n;F}(b)$ be the $n\times n$ matrix
\begin{equation}  \lb{6.1}
J_{n;F}(b) = \begin{pmatrix}
b_1 & a_1 & 0 \\
a_1 & b_2 & a_2 \\
0 & a_2 & b_3 \\
{} & {} & {} & \ddots \\
{} & {} & {} & {} & b_n +b
\end{pmatrix}
\end{equation}
(i.e., we truncate the infinite Jacobi matrix and change only the corner matrix element $b_n$ to
$b_n+b$).

Let $\ti x_j^{(n)}(b)$, $j=1, \dots, n$, be the eigenvalues of $J_{n;F}(b)$ labelled by $\ti x_1
< \ti x_2 < \dots$. (We shall shortly see these eigenvalues are all simple.) Let $\ti\varphi_j^{(n)}$
be the normalized eigenvectors with components $[\ti\varphi_j^{(n)}(b)]_\ell$, $\ell=1, \dots, n$,
and define
\begin{equation}  \lb{6.2}
\ti\lambda_j^{(n)}(b) = \abs{[\ti\varphi_j^{(n)}(b)]_1}^2
\end{equation}
so that if $e_1$ is the vector $(1\, 0 \dots 0)^t$, then
\begin{equation}  \lb{6.3}
\sum_{j=1}^n \ti\lambda_j^{(n)}(b) \delta_{\ti x_j^{(n)}(b)}
\end{equation}
is the spectral measure for $J_{n;F}(b)$ and $e_1$, that is,
\begin{equation}  \lb{6.4}
\jap{e_1, J_{n;F}(b)^\ell e_1} = \sum_{j=1}^n \ti\lambda_j^{(n)} (b) \ti x_j^{(n)}(b)^\ell
\end{equation}
for all $\ell$. We are going to begin by proving an intermediate quadrature formula:

\begin{theorem}\lb{T6.1} Let $\mu$ be a probability measure. For any $b$ and any $\ell =0,1,
\dots,2n-2$,
\begin{equation}  \lb{6.5}
\int x^\ell\, d\mu = \sum_{j=1}^n \ti\lambda_j^{(n)}(b) \ti x_j^{(n)}(b)^\ell
\end{equation}
If $b=0$, this holds also for $\ell=2n-1$.
\end{theorem}

\begin{proof} For any measure, $\{a_j,b_j\}_{j=1}^{n-1}$ determine $\{p_j\}_{j=0}^{n-1}$, and
moreover,
\begin{equation}  \lb{6.6}
\int x\abs{p_{n-1}(x)}^2 \, d\mu = b_n
\end{equation}
If a measure has finite support with at least $n$ points, one can still define $\{p_j\}_{j=0}^{n-1}$,
Jacobi parameters $\{a_j,b_j\}_{j=1}^{n-1}$, and $b_n$ by \eqref{6.6}.

$d\mu$ and the measure, call it $d\ti\mu_1^{(n)}$, of \eqref{6.3} have the same Jacobi parameters
$\{a_j,b_j\}_{j=1}^{n-1}$, so the same $\{p_j\}_{j=0}^{n-1}$, and thus by
\begin{equation}  \lb{6.7}
\int x^k p_j(x)\, d\mu = 0 \qquad k=0,1,\dots, j-1; \, j=1, \dots, n-1
\end{equation}
we inductively get \eqref{6.5} for $\ell=0,1,2,\dots, 2n-3$. Moreover,
\begin{equation}  \lb{6.8}
\int p_{n-1}(x)^2\, d\mu =1
\end{equation}
determines inductively \eqref{6.5} for $\ell = 2n-2$. Finally, if $b=0$, \eqref{6.6} yields \eqref{6.5}
for $\ell =2n-1$.
\end{proof}

As the second step, we want to determine the $\ti x_j^{(n)}(b)$ and $\ti\lambda_j^{(n)}(b)$.

\begin{theorem}\lb{T6.2} Let $K_{n;F}=\pi_{n-1} M_z \pi_{n-1} \restriction \ran(\pi_{n-1})$ for
a general finite moment measure, $\mu$, on $\bbC$. Then
\begin{equation}  \lb{6.9}
\det_{\ran(\pi_{n-1})} (z\bdone -K_{n;F}) = X_n(z)
\end{equation}
\end{theorem}

\begin{proof} Suppose $X_n(z)$ has a zero of order $\ell$ at $z_0$. Let $\varphi = X_n(z)/(z-z_0)^\ell$.
Then, in $\ran(\pi_n)$,
\begin{align}
(K_{n;F}-z_0)^j \varphi &\neq 0 \qquad j=0,1,\dots, \ell-1  \lb{6.10} \\
(K_{n;F}-z_0)^\ell \varphi &= 0 \lb{6.11}
\end{align}
since $(M_z-z_0)^\ell \varphi = X_n(z)$ and $\pi_{n-1} X_n =0$. Thus, $z_0$ is an eigenvalue of
$K_{n;F}$ of algebraic multiplicity at least $\ell$. Since $X_n(z)$ has $n$ zeros counting multiplicity,
this accounts for all the roots, so \eqref{6.9} holds because both sides of monic polynomials of degree
$n$ with the same roots.
\end{proof}

\begin{corollary} \lb{C6.3} We have for OPRL
\begin{equation}  \lb{6.5x}
\det(z-J_{n;F}(b)) = P_n(z) -b P_{n-1}(z)
\end{equation}
The eigenvalues $\ti x_j^{(n)}(b)$ are all simple and obey for $0<b<\infty$ and $j=1, \dots, n$
{\rm{(}}with $\ti x_{n+1}(0)=\infty${\rm{)}},
\begin{equation}  \lb{6.6x}
\ti x_j^{(n)}(0) < \ti x_j^{(n)}(b) < \ti x_{j+1}^{(n)}(0)
\end{equation}
and for $-\infty < b < 0$ and $j=1, \dots, n$ {\rm{(}}with $\ti x_{n-1}(0) =-\infty${\rm{)}},
\begin{equation} \lb{6.7x}
\ti x_{j-1}^{(n)}(0) < \ti x_j^{(n)}(b) < \ti x_j^{(n)}(0)
\end{equation}
\end{corollary}

\begin{proof} \eqref{6.5x} for $b=0$ is just \eqref{6.9}. Expanding in minors shows the determinant of
$(z-J_{n;F}(b))$ is just the value at $b=0$ minus $b$ times the $(n-1)\times (n-1)$ determinant, proving
\eqref{6.5x} in general.

The inequalities in \eqref{6.6x}/\eqref{6.7x} follow either by eigenvalue perturbation theory or by
using the arguments in Section~\ref{s4}.
\end{proof}

In fact, our analysis below proves that for $0<b<\infty$,
\begin{equation}  \lb{6.8x}
\ti x_j^{(n)}(0) < \ti x_j^{(n)}(b) < \ti x_j^{(n-1)}(0)
\end{equation}

The recursion formula for monic OPs proves that $p_j(\ti x_j(b))$ is the unnormalized eigenvector for
$J_{n;F}(b)$. $K_{n-1} (\ti x_j(b), \ti x_j(b))^{1/2}$ is the normalization constant, so since $p_0
\equiv 1$ (if $\mu(\bbR)=1$):

\begin{proposition}\lb{P6.4} If $\mu(\bbR)=1$, then
\begin{equation}  \lb{6.9x}
\lambda_j^{(n)}(b) = (K_{n-1} (\ti x_j^{(n)}(b), \ti x_j^{(n)}(b)))^{-1}
\end{equation}
\end{proposition}

Now fix $n$ and $x_0\in\bbR$. Define
\begin{equation}  \lb{6.10x}
b(x_0) = \f{P_n(x_0)}{P_{n-1}(x_0)}
\end{equation}
with the convention $b=\infty$ if $P_{n-1}(x_0)=0$. Define for $b\neq \infty$,
\begin{equation}  \lb{6.11x}
x_j^{(n)}(x_0) = \ti x_j^{(n)}(b(x_0)) \qquad j=1, \dots, n
\end{equation}
and if $b(x_0)=\infty$,
\begin{equation}  \lb{6.12}
x_j^{(n)}(x_0) = \ti x_j^{(n-1)} (0) \qquad j=1, \dots, n-1
\end{equation}
and
\begin{equation}  \lb{6.13}
\lambda_j^{(n)}(x_0) = (K_{n-1} (x_j^{(n)}(x_0), x_j^{(n)}(x_0)))^{-1}
\end{equation}
Then Theorem~\ref{T6.1} becomes

\begin{theorem}[Gaussian Quadrature]\lb{T6.5} Fix $n,x_0$. Then
\begin{equation}  \lb{6.14}
\int Q(x)\, d\mu = \sum_{j=1}^n \lambda_j^{(n)}(x_0) Q(x_j^{(n)} (x_0))
\end{equation}
for all polynomials $Q$ of degree up to:
\begin{SL}
\item[{\rm{(1)}}] $2n-1$ if $P_n(x_0) =0$
\item[{\rm{(2)}}] $2n-2$ if $P_n(x_0)\neq 0 \neq P_{n-1}(x_0)$
\item[{\rm{(3)}}] $2n-3$ if $P_{n-1}(x_0)=0$.
\end{SL}
\end{theorem}

\begin{remarks} 1. The sum goes to $n-1$ if $P_{n-1}(x_0)=0$.

\smallskip
2. We can define $x_j^{(n)}$ to be the solutions of
\begin{equation}  \lb{6.15}
p_{n-1}(x_0) p_n (x) - p_n(x_0) p_{n-1}(x)=0
\end{equation}
which has degree $n$ if $p_{n-1}(x_0)\neq 0$ and $n-1$ if $p_{n-1}(x_0)=0$.

\smallskip
3. \eqref{6.13} makes sense even if $\mu(\bbR)\neq 1$ and dividing by $\mu(\bbR)$ changes
$\int Q(x)\, d\mu$ and $\lambda_j^{(n)}$ by the same amount, so \eqref{6.14} holds for all
positive $\mu$ (with finite moments), not just the normalized ones.

\smallskip
4. The weights, $\lambda_j^{(n)}(x_0)$, in Gaussian quadrature are called Cotes numbers.
\end{remarks}

\section{Markov--Stieltjes Inequalities} \lb{s7}

The ideas of this section go back to Markov \cite{Markov} and Stieltjes \cite{Stie}
based on conjectures of Chebyshev \cite{Cheb} (see Freud \cite{FrB}).

\begin{lemma}\lb{L7.1} Fix $x_1 < \cdots < x_n$ in $\bbR$ distinct and $1\leq\ell < n$. Then there is a
polynomial, $Q$, of degree $2n-2$ so that
\begin{SL}
\item[{\rm{(i)}}]
\begin{equation}  \lb{7.1}
Q(x_j) = \begin{cases} 1 & j=1, \dots, \ell \\
0 & 1=\ell+1, \dots, n
\end{cases}
\end{equation}
\item[{\rm{(ii)}}] For all $x\in\bbR$,
\begin{equation}  \lb{7.2}
Q(x)\geq \chi_{(-\infty, x_\ell]} (x)
\end{equation}
\end{SL}
\end{lemma}

\begin{remark} Figure~1  has a graph of $Q$ and $\chi_{(-\infty, x_\ell]}$ for $n=5$, $\ell=3$, $x_j=j-1$.
\end{remark}

\begin{center}
\begin{figure}[h]
\includegraphics[scale=0.6]{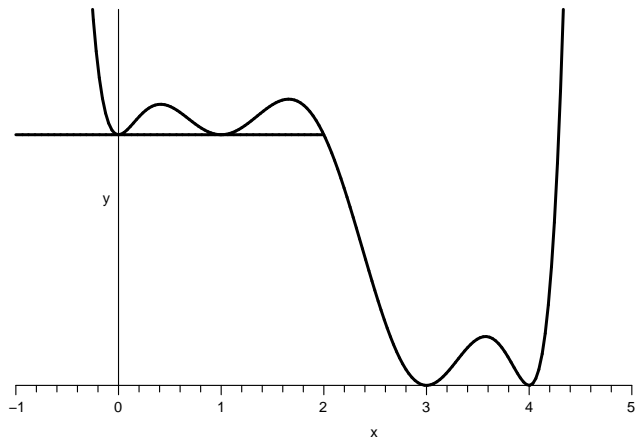}
\caption{An interpolation polynomial}
\end{figure}
\end{center}

\begin{proof} By standard interpolation theory, there exists a unique polynomial of degree $k$ with $k+1$
conditions of the form
\[
Q(y_j) = Q'(y_j) = \cdots = Q^{(n_j)}(y_j) = 0
\]
$\sum_j n_j = k+1$. Let $Q$ be the polynomial of degree $2n-2$ with the $n$ conditions in \eqref{7.1} and
the $n-1$ conditions
\begin{equation}  \lb{7.3}
Q'(x_j)=0\qquad j=1, \dots, \ell-1, \ell+1, \dots, n
\end{equation}

Clearly, $Q'$ has at most $2n-3$ zeros. $n-1$ are given by \eqref{7.3} and, by Snell's theorem, each of
the $n-2$ intervals $(x_1, x_2), \dots, (x_{\ell-1}, x_\ell), (x_{\ell+1}, x_{\ell+2}), \dots,
(x_{n-1}, x_n)$ must have a zero. Since $Q'$ is nonvanishing on $(x_\ell,x_{\ell+1})$ and $Q(x_\ell)=1
> Q(x_{\ell+1})=0$, $Q'(y)<0$ on $(x_\ell,x_{\ell+1})$. Tracking where $Q'$ changes sign, one sees that
\eqref{7.2} holds.
\end{proof}

\begin{theorem} \lb{T7.2} Suppose $d\mu$ is a measure on $\bbR$ with finite moments. Then
\begin{equation}  \lb{7.4}
\begin{split}
&\sum_{\{j\mid x_j^{(n)}(x_0) \leq x_0\}} \f{1}{K_{n-1} (x_j^{(n)} (x_0), x_j^{(n)}(x_0))}
\geq \mu((-\infty, x_0]) \\
&\qquad \geq \mu((-\infty, x_0)) \geq \sum_{\{j\mid x_j^{(n)}(x_0) < x_0\}}
\f{1}{K_{n-1} (x_j^{(n)}(x_0), x_j^{(n)}(x_0))}
\end{split}
\end{equation}
\end{theorem}

\begin{remarks} 1. The two bounds differ by $K_{n-1}(x_0,x_0)^{-1}$.

\smallskip
2. These imply
\begin{equation}  \lb{7.5}
\mu(\{x_0\}) \leq K_{n-1}(x_0,x_0)^{-1}
\end{equation}
In fact, one knows (see \eqref{9.18} below)
\begin{equation}  \lb{7.6}
\mu(\{x_0\}) =\lim_{n\to\infty}\, K_{n-1} (x_0,x_0)^{-1}
\end{equation}
If $\mu(\{x_0\}) =0$, then the bounds are exact as $n\to\infty$.
\end{remarks}

\begin{proof} Suppose $P_{n-1}(x_0)\neq 0$. Let $\ell$ be such that $x_\ell^{(n)}(x_0)=x_0$.
Let $Q$ be the polynomial of Lemma~\ref{L7.1}. By \eqref{7.2},
\[
\mu((-\infty, x_0]) \leq \int Q(x)\, d\mu
\]
and, by \eqref{7.1} and Theorem~\ref{T6.5}, the integral is the sum on the left of \eqref{7.4}.

Clearly, this implies
\[
\mu((x_0,\infty))\geq \sum_{\{j\mid x_j^{(n)}(x_0) > x_0\}} \f{1}{K_{n-1}(x_j^{(n)}(x_0),
x_j^{(n)}(x_0))}
\]
which, by $x\to -x$ symmetry, implies the last inequality in \eqref{7.4}.
\end{proof}

\begin{corollary} If $\ell \leq k-1$, then
\begin{equation}  \lb{7.7}
\begin{aligned}
\sum_{j=\ell+1}^{k-1} \f{1}{K(x_j^{(n)}(x_0), x_j^{(n)}(x_0))} & \leq \mu ([x_\ell^{(n)}(x_0),
x_k^{(n)}(x_0)]) \\
& \leq \sum_{j=\ell}^k \f{1}{K(x_j^{(n)}(x_0), x_j^{(n)}(x_0))}
\end{aligned}
\end{equation}
\end{corollary}

\begin{proof} Note if $x_1=x_\ell^{(n)}(x_0)$ for some $\ell$, then $x_j^{(n)}(x_0)=x_j^{(n)}
(x_1)$, so we get \eqref{7.7} by subtracting values of \eqref{7.4}.
\end{proof}

Notice that this corollary gives effective lower bounds only if $k-1\geq\ell+1$, that is, only on
at least three consecutive zeros. The following theorem of Last--Simon \cite{S309}, based on ideas
of Golinskii \cite{Gol02}, can be used on successive zeros (see \cite{S309} for the proof).

\begin{theorem}\lb{T7.4} If $E,E'$ are distinct zeros of $P_n(x)$, $\wti E=\f12(E+E')$ and
$\delta >\f12\abs{E-E'}$, then
\begin{equation}  \lb{7.8}
\abs{E-E'} \geq \f{\delta^2 - (\f12\abs{E-E'}^2)^2}{3n}
\biggl[ \f{K_n(E,E)}{\sup_{\abs{y-\wti E} \leq \delta} K_n(y,y)}\biggr]^{1/2}
\end{equation}
\end{theorem}

\section{Mixed CD Kernels} \lb{s8}

Recall that given a measure $\mu$ on $\bbR$ with finite moments and Jacobi parameters
$\{a_n,b_n\}_{n=1}^\infty$, the {\it second kind polynomials} are defined by the recursion
relations \eqref{1.5} but with initial conditions
\begin{equation}  \lb{2.1.1}
q_0(x) =0 \qquad q_1(x)=a_1^{-1}
\end{equation}
so $q_n(x)$ is a polynomial of degree $n-1$. In fact, if $\ti\mu$ is the measure with Jacobi
parameters given by
\[
\ti a_n =a_{n+1} \qquad \ti b_n=b_{n+1}
\]
then
\begin{equation}  \lb{8.1}
q_n(x;d\mu) = a_1^{-1} p_{n-1} (x;d\ti\mu)
\end{equation}
It is sometimes useful to consider
\begin{equation}  \lb{8.2}
K_n^{(q)} (x,y) = \sum_{j=0}^n\, \ol{q_j(x)}\, q_j(y)
\end{equation}
and the mixed CD kernel
\begin{equation}  \lb{8.3}
K_n^{(pq)} (x,y) = \sum_{j=0}^n\, \ol{q_j(x)}\, p_j(y)
\end{equation}

Since \eqref{8.1} implies
\begin{equation}  \lb{8.4}
K_n^{(q)} (x,y;d\mu) = a_1^{-2} K_{n-1} (x,y;d\ti\mu)
\end{equation}
there is a CD formula for $K^{(q)}$ which follows immediately from the one for $K$. There is also a
mixed CD formula for $K_n^{(pq)}$.

OPUC also have second kind polynomials, mixed CD kernels, and mixed CD formulae. These are discussed in
Section~3.2 of \cite{OPUC1}.

Mixed CD kernels will enter in Section~\ref{s21}.

\section{Variational Principle: Basics} \lb{s9}

If one thing marks the OP approach to the CD kernel that has been missing from the spectral
theorists' approach, it is a remarkable variational principle for the diagonal kernel. We
begin with:

\begin{lemma}\lb{L9.1} Fix $(\alpha_1, \dots, \alpha_m)\in\bbC^m$. Then
\begin{equation}  \lb{9.1}
\min\biggl(\, \sum_{j=1}^m \abs{z_j}^2 \biggm| \sum_{j=1}^m \alpha_j z_j =1\biggr) =
\biggl( \, \sum_{j=1}^m \, \abs{\alpha_j}^2\biggr)^{-1}
\end{equation}
with the minimizer given uniquely by
\begin{equation}  \lb{9.2}
z_j^{(0)} = \f{\bar\alpha_j}{\sum_{j=1}^m\, \abs{\alpha_j}^2}
\end{equation}
\end{lemma}

\begin{remark} One can use Lagrange multipliers to a priori compute $z_j^{(0)}$ and prove this result.
\end{remark}

\begin{proof} If
\begin{equation}  \lb{9.3}
\sum_{j=1}^m \alpha_j z_j =1
\end{equation}
then
\begin{equation}  \lb{9.4}
\sum_{j=1}^m\, \abs{z_j-z_j^{(0)}}^2 = \sum_{j=1}^m \, \abs{z_j}^2 - \biggl( \, \sum_{j=1}^m \,
\abs{\alpha_j}^2\biggr)^{-1}
\end{equation}
from which the result is obvious.
\end{proof}

If $Q$ has $\deg(Q)\leq n$ and $Q_n (z_0)=1$, then
\begin{equation}  \lb{9.5}
Q_n(z) =\sum_{j=0}^n \alpha_j x_j (z)
\end{equation}
with $x_j$ the orthonormal polynomials for a measure $d\mu$, then $\sum\alpha_j x_j (z_0)=1$ and
$\Norm{Q_n}_{L^2(\bbC,d\mu)}^2 = \sum_{j=0}^n \abs{\alpha_j}^2$. Thus the lemma implies:

\begin{theorem}[Christoffel Variational Principle]\lb{T9.2} Let $\mu$ be a measure on $\bbC$ with finite
moments. Then for $\bar z_0\in\bbC$,
\begin{equation}  \lb{9.6}
\min\biggl(\int \abs{Q_n(z)}^2\, d\mu \biggm| Q_n(z_0) =1, \, \deg(Q_n)\leq n\biggr)
= \f{1}{K_n(z_0,z_0)}
\end{equation}
and the minimizer is given by
\begin{equation}  \lb{9.7}
Q_n(z,z_0) = \f{K_n(z_0,z)}{K_n(z_0,z_0)}
\end{equation}
\end{theorem}

One immediate useful consequence is:

\begin{theorem}\lb{T9.3} If $\mu\leq \nu$, then
\begin{equation}  \lb{9.8}
K_n(z,z;d\nu) \leq K_n(z,z;d\mu)
\end{equation}
\end{theorem}

For this reason, it is useful to have comparison models:

\begin{example}\lb{E9.4} Let $d\mu = d\theta/2\pi$ for $z=re^{i\theta}$ and $\zeta =
e^{i\varphi}$. We have, since $\varphi_n(z)=z^n$,
\begin{equation}  \lb{9.9}
K_n(z,\zeta) = \f{1-r^{n+1} e^{i(n+1)(\varphi-\theta)}}{1-re^{i(\varphi-\theta)}}
\end{equation}
If $r<1$, $K_n(z,z_0)$ has a limit as $n\to\infty$, and for $z=e^{i\varphi}$, $z_0 =
re^{i\theta}$, $r<1$,
\begin{equation}  \lb{9.10}
\abs{Q_n(z,z_0)}^2 \, \f{d\varphi}{2\pi} \to P_r (\theta,\varphi)\, \f{d\varphi}{2\pi}
\end{equation}
the Poisson kernel,
\begin{equation}  \lb{9.11}
P_r (\theta,\varphi) = \f{1-r^2}{1+r^2 - 2r\cos(\theta-\varphi)}
\end{equation}
For $r=1$, we have
\begin{equation}  \lb{9.12}
\abs{K_n(e^{i\theta}, e^{i\varphi})}^2 = \f{\sin^2 (\f{n+1}{2}(\theta-\varphi))}{\sin^2 (\theta-\varphi)}
\end{equation}
the Fej\'er kernel.

For $r>1$, we use
\begin{equation}  \lb{9.12a}
K_n(z,\zeta) = \bar z^n \zeta^n K_n \biggl( \f{1}{z}\, , \f{1}{\zeta}\biggr)
\end{equation}
which implies, for $z=e^{i\varphi}$, $z_0=re^{i\theta}$, $r>1$,
\begin{equation}  \lb{9.12b}
\abs{Q_n(z,z_0)}^2 \, \f{d\varphi}{2\pi} \to P_{r^{-1}} (\theta,\varphi)\, \f{d\varphi}{2\pi}
\end{equation}
\qed
\end{example}

\begin{example}\lb{E9.5} Let $d\mu_0$ be the measure
\begin{equation}  \lb{9.13}
d\mu_0(x) = \f{1}{2\pi}\, \sqrt{4-x^2}\, \chi_{[-2,2]}(x)\, dx
\end{equation}
on $[-2,2]$. Then $p_n$ are the Chebyshev polynomials of the second kind,
\begin{equation}  \lb{9.14}
p_n(2\cos\theta) = \f{\sin(n+1)\theta}{\sin\theta}
\end{equation}
In particular, if $\abs{x}\leq 2-\delta$,
\begin{equation}  \lb{9.15}
\abs{p_n(x+iy)} \leq C_{1,\delta} e^{nC_{2,\delta}\abs{y}}
\end{equation}
and so
\begin{equation}  \lb{9.16}
\f{1}{n}\, \abs{K_n(x+iy,x+iy)} \leq C_{1,\delta}^2 e^{2n C_{2,\delta}\abs{y}}
\end{equation}
\qed
\end{example}

The following shows the power of the variational principle:

\begin{theorem}\lb{T9.4} Let
\begin{equation}  \lb{9.16a}
d\mu = w(x)\, dx + d\mu_\s
\end{equation}
Suppose for some $x_0,\delta$, we have
\[
w(x) \geq c>0
\]
for $x\in [x_0-\delta, x_0 + \delta]$. Then for any $\delta' <\delta$ and all $x\in [x_0-\delta',
x_0 + \delta']$, we have for all $a$ real,
\begin{equation}  \lb{9.17}
\f{1}{n}\, K_n \biggl( x+ \f{ia}{n}, x + \f{ia}{n}\biggr) \leq C_1 e^{C_2 \abs{a}}
\end{equation}
\end{theorem}

\begin{proof} We can find a scaled and translated version of the $d\mu_0$ of \eqref{9.13} with $\mu\geq\mu_0$.
Now use Theorem~\ref{T9.3} and \eqref{9.16}.
\end{proof}

The following has many proofs, but it is nice to have a variational one:

\begin{theorem}\lb{T9.5} Let $\mu$ be a measure on $\bbR$ of compact support. For all $x_0\in\bbR$,
\begin{equation}  \lb{9.18}
\lim_{n\to\infty} \, K_n (x_0,x_0) = \mu(\{x_0\})^{-1}
\end{equation}
\end{theorem}

\begin{remark} If $\mu(\{x_0\})=0$, the limit is infinite.
\end{remark}

\begin{proof} Clearly, if $Q(x_0)=1$, $\int \abs{Q_n(x)}^2\, d\mu \geq \mu(\{x_0\})$, so
\begin{equation}  \lb{9.19}
K_n(x_0,x_0) \leq \mu(\{x_0\})^{-1}
\end{equation}
On the other hand, pick $A\geq\diam(\sigma(d\mu))$ and let
\begin{equation}  \lb{9.20}
Q_{2n}(x) = \biggl( 1-\f{(x-x_0)^2}{A^2}\biggr)^n
\end{equation}
For any $a$,
\begin{equation}  \lb{9.21}
\sup_{\substack{\abs{x-x_0} \geq a \\ x\in\sigma(d\mu)}} \abs{Q_{2n}(x)}\equiv M_{2n}(a) \to 0
\end{equation}
so, since $Q_{2n}\leq 1$ on $\sigma(d\mu)$,
\begin{equation}  \lb{9.22}
K_n(x_0,x_0)\geq [\mu((x_0-a, x_0+a))+ M_{2n}(a)]^{-1}
\end{equation}
so
\begin{equation}  \lb{9.23}
\liminf\, K_n(x_0,x_0) \geq [\mu((x_0-a, x_0 +a))]
\end{equation}
for each $a$. Since $\lim_{a\downarrow 0} \mu((x_0-a, x_0+a))=\mu(\{x_0\})$,
\eqref{9.19} and \eqref{9.23} imply \eqref{9.18}.
\end{proof}

\section{The Nevai Class: An Aside} \lb{s10}

In his monograph, Nevai \cite{Nev79} emphasized the extensive theory that can be developed for OPRL
measures whose Jacobi parameters obey
\begin{equation}  \lb{10.1}
a_n\to a \qquad b_n\to b
\end{equation}
for some $b$ real and $a>0$. He proved such measures have ratio asymptotics, that is, $P_{n+1}
(z)/P_n(z)$ has a limit for all $z\in\bbC\setminus\bbR$, and Simon \cite{S290} proved a converse:
Ratio asymptotics at one point of $\bbC_+$ implies there are $a,b,$ with \eqref{10.1}. The
essential spectrum for such a measure is $[b-2a,b+2a]$, so the Nevai class is naturally
associated with a single interval $\fre\subset\bbR$.

The question of what is the proper analog of the Nevai class for a set $\fre$ of the form
\begin{equation}  \lb{10.2}
\fre=[\alpha_1,\beta_1]\cup [\alpha_2,\beta_2]\cup \dots [\alpha_{\ell+1},\beta_{\ell+1}]
\end{equation}
with
\begin{equation}  \lb{10.3}
\alpha_1 <\beta_1 < \cdots < \alpha_{\ell+1} < \beta_{\ell+1}
\end{equation}
has been answered recently and is relevant below.

The key was the realization of L\'opez \cite{BRLL,BHLL} that the proper analog of an arc of a
circle was $\abs{\alpha_n}\to a$ and $\bar\alpha_{n+1} \alpha_n \to a^2$ for some $a>0$. This
is not that $\alpha_n$ approaches a fixed sequence but rather that for each $k$,
\begin{equation}  \lb{10.4}
\min_{e^{i\theta}\in\partial\bbD}\, \sum_{j=n}^{n+k} \, \abs{\alpha_j -ae^{i\theta}} \to 0
\end{equation}
as $n\to\infty$. Thus, $\alpha_j$ approaches a set of Verblunsky coefficients rather than a
fixed one.

For any finite gap set $\fre$ of the form \eqref{10.2}/\eqref{10.3}, there is a natural torus,
$\calJ_\fre$, of almost periodic Jacobi matrics with $\sigma_\ess (J) =\fre$ for all $J\in\calJ_\fre$.
This can be described in terms of minimal Herglotz functions \cite{SY,Rice} or reflectionless
two-sided Jacobi matrices \cite{Rem}. All $J\in\calJ_\fre$ are periodic if and only if each
$[\alpha_j,\beta_j]$ has rational harmonic measure. In this case, we say $\fre$ is periodic.

\begin{definition}
\begin{gather}
d_m(\{a_n,b_n\}_{n=1}^\infty, \{\ti a_n,\ti b_n\}_{n=1}^\infty) =
\sum_{j=0}^\infty e^{-j} (\abs{a_{m+j} - \ti a_{m+j}} + \abs{b_{m+j} - \ti b_{m+j}}) \lb{10.5} \\
d_m (\{a_n,b_n\},\calJ_\fre) = \min_{J\in\calJ_\fre}\, d_m (\{a_n,b_n\},J) \lb{10.6}
\end{gather}
\end{definition}

\begin{definition} The {\it Nevai class\/} for $\fre$, $N(\fre)$, is the set of all Jacobi matrices,
$J$, with
\begin{equation}  \lb{10.7}
d_m (J,\calJ_\fre)\to 0
\end{equation}
as $m\to\infty$.
\end{definition}

This definition is implicit in Simon \cite{OPUC2}; the metric $d_m$ is from \cite{DKS}. Notice that
in case of a single gap $\fre$ in $\partial\bbD$, the isospectral torus is the set of
$\{\alpha_n\}_{n=0}^\infty$ with $\alpha_n =ae^{i\theta}$ for all $n$ where $a$ is $\fre$ dependent
and fixed and $\theta$ is arbitrary. The above definition is the L\'opez class.

That this is the ``right'' definition is seen by the following pair of theorems:

\begin{theorem}[Last--Simon \cite{S304}]\lb{T10.1} If $J\in N(\fre)$, then
\begin{equation}  \lb{10.8}
\sigma_\ess (J)=\fre
\end{equation}
\end{theorem}

\begin{theorem}[\cite{DKS} for periodic $\fre$'s; \cite{Rem} in general]\lb{T10.2} If
\[
\sigma_\ess (J)=\sigma_\ac (J)=\fre
\]
then $J\in N(\fre)$.
\end{theorem}

\section{Delta Function Limits of Trial Polynomials} \lb{s11}

Intuitively, the minimizer, $Q_n(x,x_0)$, in the Christoffel variational principle must be $1$ at
$z_0$ and should try to be small on the rest of $\sigma(d\mu)$. As the degree gets larger and
larger, one expects it can do this better and better. So one might guess that for every $\delta >0$,
\begin{equation}  \lb{11.1}
\sup_{\substack{\abs{x-x_0}>\delta \\ x\in \sigma(d\mu)}} \abs{Q_n(x,x_0)} \to 0
\end{equation}
While this happens in many cases, it is too much to hope for. If $x_1\in\sigma(d\mu)$ but $\mu$ has
very small weight near $x_1$, then it may be a better strategy for $Q_n$ not to be small very near
$x_1$. Indeed, we will see (Example~\ref{E11.3}) that the $\sup$ in \eqref{11.1} can go to infinity.
What is more likely is to expect that $\abs{Q_n(x,x_0)}^2\, d\mu$ will be concentrated near $x_0$.
We normalize this to define
\begin{equation}  \lb{11.2}
d\eta_n^{(x_0)}(x) = \f{\abs{Q_n(x,x_0)}^2\, d\mu(x)}{\int \abs{Q_n(x,x_0)}^2\, d\mu(x)}
\end{equation}
so, by \eqref{9.6}/\eqref{9.7}, in the OPRL case,
\begin{equation}  \lb{11.3}
d\eta_n^{(x_0)}(x) = \f{\abs{K_n(x,x_0)}^2}{K_n(x,x_0)}\, d\mu(x)
\end{equation}
We say $\mu$ obeys the Nevai $\delta$-convergence criterion if and only if, in the sense of
weak (aka vague) convergence of measures,
\begin{equation}  \lb{11.4}
d\eta_n^{(x_0)}(x)\to \delta_{x_0}
\end{equation}
the point mass at $x_0$. In this section, we will explore when this holds.

Clearly, if $x_0\notin\sigma(d\mu)$, \eqref{11.4} cannot hold. We saw, for OPUC with $d\mu = d\theta/2\pi$
and $z\notin\partial\bbD$, the limit was a Poisson measure, and similar results should hold for suitable
OPRL. But we will see below (Example~\ref{E11.2}) that even on $\sigma(d\mu)$, \eqref{11.4} can fail.
The major result below is that for Nevai class on $\fre^\intt$, it does hold. We begin with an equivalent
criterion:

\begin{definition} We say Nevai's lemma holds if
\begin{equation}  \lb{11.5}
\lim_{n\to\infty}\, \f{\abs{p_n(x_0)}^2}{K_n(x_0,x_0)} =0
\end{equation}
\end{definition}

\begin{theorem}\lb{T11.1} If $d\mu$ is a measure on $\bbR$ with bounded support and
\begin{equation}  \lb{11.6}
\inf_n\, a_n >0
\end{equation}
then for any fixed $x_0\in\bbR$,
\[
\eqref{11.4} \Leftrightarrow \eqref{11.5}
\]
\end{theorem}

\begin{remark} That \eqref{11.5} $\Rightarrow$ \eqref{11.4} is in Nevai \cite{Nev79}.
The equivalence is a result of Breuer--Last--Simon \cite{BLSprep}.
\end{remark}

\begin{proof} Since
\begin{gather}
1 - \f{K_{n-1}(x_0,x_0)}{K_n(x_0,x_0)} = \f{\abs{p_n(x_0)}^2}{K_n(x_0,x_0)}  \lb{11.7} \\
\eqref{11.5} \Leftrightarrow \f{K_{n-1}(x_0,x_0)}{K_n(x_0,x_0)} \to 1 \lb{11.8}
\end{gather}
so
\[
\eqref{11.5} \Rightarrow \f{\abs{p_{n+1}(x_0)}^2}{K_n(x_0,x_0)} =
\f{\abs{p_{n+1}(x_0)}^2}{K_{n+1}(x_0,x_0)}\, \f{K_{n+1}(x_0,x_0)}{K_n(x_0,x_0)}\to 0
\]
We thus conclude
\begin{equation}  \lb{11.9}
\eqref{11.5} \Leftrightarrow \f{\abs{p_n(x_0)}^2+ \abs{p_{n+1}(x_0)}^2}{K_n(x_0,x_0)} \to 0
\end{equation}

By the CD formula and orthonormality of $p_j(x)$,
\begin{equation}  \lb{11.9x}
\int\abs{x-x_0}^2 \abs{K_n(x,x_0)}^2\, d\mu = a_{n+1}^2 [p_n(x_0)^2 + p_{n+1}(x_0)^2]
\end{equation}
so, by \eqref{11.6} and \eqref{11.9x},
\[
\int \abs{x-x_0}^2 \, d\eta_n^{(x_0)}(x) \to 0 \Leftrightarrow \eqref{11.5}
\]
when $a_n$ is uniformly bounded above and away from zero. But since $d\eta_n$ have support in a
fixed interval,
\[
\eqref{11.4} \Leftrightarrow \int \abs{x-x_0}^2\, d\eta_n^{(x_0)} \to 0
\qedhere
\]
\end{proof}

\begin{example}\lb{E11.2} Suppose at some point $x_0$, we have
\begin{equation}  \lb{11.10}
\lim_{n\to\infty}\, (\abs{p_n(x_0)}^2 + \abs{p_{n+1}(x_0)}^2)^{1/n} \to A>1
\end{equation}
We claim that
\begin{equation}  \lb{11.11}
\limsup_{n\to\infty}\, \f{\abs{p_n(x_0)}^2}{K_n(x_0,x_0)} >0
\end{equation}
for if \eqref{11.11} fails, then \eqref{11.5} holds and, by \eqref{11.7}, for any $\veps$,
we can find $N_0$ so for $n\geq N_0$,
\begin{equation}  \lb{11.12}
K_{n+1} (x_0,x_0) \leq (1+\veps) K_n (x_0,x_0)
\end{equation}
so
\[
\lim\, K_n(x_0,x_0)^{1/n} \leq 1
\]
So, by \eqref{11.5}, \eqref{11.10} fails. Thus, \eqref{11.10} implies that \eqref{11.5} fails,
and so \eqref{11.4} fails.
\qed
\end{example}

\begin{remark} As the proof shows, rather than a limit in \eqref{11.11}, we can have a
$\liminf > 1$.
\end{remark}

The first example of this type was found by Szwarc \cite{Szwarc}. He has a $d\mu$ with pure points
at $2-n^{-1}$ but not at $2$, and so that the Lyapunov exponent at $2$ was positive but $2$ was not
an eigenvalue, so \eqref{11.10} holds. The Anderson model (see \cite{CarLac}) provides a more
dramatic example. The spectrum is an interval $[a,b]$ and \eqref{11.10} holds for a.e.\ $x\in [a,b]$.
The spectral measure in this case is supported at eigenvalues and at eigenvalues \eqref{11.8},
and so \eqref{11.4} holds. Thus \eqref{11.4} holds on a dense set in $[a,b]$ but fails for
Lebesgue a.e.\ $x_0$!

\begin{example}\lb{E11.3} A Jacobi weight has the form
\begin{equation}  \lb{11.12x}
d\mu(x) = C_{a,b} (1-x)^a (1+x)^b\, dx
\end{equation}
with $a,b >-1$. In general, one can show \cite{Szb}
\begin{equation}  \lb{11.13}
p_n(1)\sim cn^{a+1/2}
\end{equation}
so if $x_0\in (-1,1)$ where $\abs{p_n(x_0)}^2 + \abs{p_{n-1}(x_0)}^2$ is bounded above and below, one has
\[
\f{\abs{K_n(x_0,1)}}{K_n(x_0,x_0)} \sim \f{n^{a+1/2}}{n} = n^{a-1/2}
\]
so if $a>\f12$, $\abs{Q_n(x_0,1)}\to\infty$. Since $d\mu(x)$ is small for $x$ near $1$, one can
(and, as we will see, does) have \eqref{11.4} even though \eqref{11.1} fails.
\qed
\end{example}

With various counterexamples in place (and more later!), we turn to the positive results:

\begin{theorem}[Nevai \cite{Nev79}, Nevai--Totik--Zhang \cite{NTZ91}]\lb{T11.4} If $d\mu$ is a
measure in the classical Nevai class {\rm{(}}i.e., for a single interval, $\fre=[b-2a,b+2a]${\rm{)}},
then \eqref{11.5} and so \eqref{11.4} holds uniformly on $\fre$.
\end{theorem}

\begin{theorem}[Zhang \cite{Zhang}, Breuer--Last--Simon \cite{BLSprep}]\lb{T11.5} Let $\fre$ be a periodic
finite gap set and let $\mu$ lie in the Nevai class for $\fre$. Then \eqref{11.5} and so \eqref{11.4} holds
uniformly on $\fre$.
\end{theorem}

\begin{theorem}[Breuer--Last--Simon \cite{BLSprep}]\lb{T11.6} Let $\fre$ be a general finite gap set
and let $\mu$ lie in the Nevai class for $\fre$. Then \eqref{11.5} and so \eqref{11.4} holds uniformly
on compact subsets of $\fre^\intt$.
\end{theorem}

\begin{remarks} 1. Nevai \cite{Nev79} proved \eqref{10.4}/\eqref{10.5} for the classical Nevai class for
every energy in $\fre$ but only uniformly on compacts of $\fre^\intt$. Uniformity on all of $\fre$
using a beautiful lemma is from \cite{NTZ91}.

\smallskip
2. Zhang \cite{Zhang} proved Theorem~\ref{T11.5} for any $\mu$ whose Jacobi parameters approached a fixed
periodic Jacobi matrix. Breuer--Last--Simon \cite{BLSprep} noted that without change, Zhang's result
holds for the Nevai class.

\smallskip
3. It is hoped that the final version of \cite{BLSprep} will prove the result in Theorem~\ref{T11.6}
on all of $\fre$, maybe even uniformly in $\fre$.
\end{remarks}

\begin{example}[\cite{BLSprep}]\lb{E11.7} In the next section, we will discuss regular measures. They have
zero Lyapunov exponent on $\sigma_\ess(\mu)$, so one might expect Nevai's lemma could hold---and it will
in many regular cases. However, \cite{BLSprep} prove that if $b_n\equiv 0$ and $a_n$ is alternately $1$
and $\f12$ on successive very long blocks ($1$ on blocks of size $3^{n^2}$ and $\f12$ on blocks of
size $2^{n^2}$), then $d\mu$ is regular for $\sigma(d\mu) =[-2,2]$. But for a.e.\ $x\in [-2,2]
\setminus [-1,1]$, \eqref{10.4} and \eqref{10.3} fail.
\qed
\end{example}

\begin{conjecture}[\cite{BLSprep}]\lb{Con11.8} The following is extensively discussed in \cite{BLSprep}:
For general OPRL of compact support and a.e.\ $x$ with respect to $\mu$, \eqref{10.4} and so \eqref{10.3}
holds.
\end{conjecture}

\section{Regularity: An Aside} \lb{s12}

There is another class besides the Nevai class that enters in variational problems because it allows
exponential bounds on trial polynomials. It relies on notions from potential theory; see \cite{Helms,Land,Ran,Tsu}
for the general theory and \cite{StT,EqMC} for the theory in the context of orthogonal polynomials.

\begin{definition} Let $\mu$ be a measure with compact support and let $\fre=\sigma_\ess(\mu)$.
We say $\mu$ is {\it regular\/} for $\fre$ if and only if
\begin{equation} \lb{12.1}
\lim_{n\to\infty}\, (a_1\dots a_n)^{1/n} = C(\fre)
\end{equation}
the capacity of $\fre$.
\end{definition}

For $\fre=[-1,1]$, $C(\fre)=\f12$ and the class of regular measures was singled out initially by
Erd\H{o}s--Tur\'an \cite{ET} and extensively studied by Ullman \cite{Ull}. The general theory was developed
by Stahl--Totik \cite{StT}.

Recall that any set of positive capacity has an equilibrium measure, $\rho_\fre$, and Green's function,
$G_\fre$, defined by requiring $G_\fre$ is harmonic on $\bbC\setminus\fre$, $G_\fre(z)=\log\abs{z} + O(1)$
near infinity, and for quasi-every $x\in\fre$,
\begin{equation} \lb{12.2}
\lim_{z_n\to x}\, G_\fre (z_n)=0
\end{equation}
(quasi-every means except for a set of capacity $0$). $\fre$ is called regular for the Dirichlet problem
if and only if \eqref{12.2} holds for every $x\in\fre$. Finite gap sets are regular for the Dirichlet
problem.

One major reason regularity will concern us is:

\begin{theorem}\lb{T12.1} Let $\fre\subset\bbR$ be compact and regular for the Dirichlet problem. Let
$\mu$ be a measure regular for $\fre$. Then for any $\veps$, there is $\delta >0$ and $C_\veps$ so that
\begin{equation} \lb{12.2x}
\sup_{\dist(z,\fre)<\delta}\, \abs{p_n(z,d\mu)} \leq C_\veps e^{\veps\abs{n}}
\end{equation}
\end{theorem}

For proofs, see \cite{StT,EqMC}. Since $K_n$ has $n+1$ terms, \eqref{12.2x} implies
\begin{equation} \lb{12.3}
\sup_{\substack{\dist(z,\fre) <\delta \\ \dist(w,\fre)<\delta}} \abs{K_n(z,w)} \leq (n+1)
C_\veps^2 e^{2\veps\abs{n}}
\end{equation}
and for the minimum (since $K_n(z_0,z_0)\geq 1$),
\begin{equation} \lb{12.4}
\sup_{\substack{\dist(z,\fre) <\delta \\ \dist(z_0,\fre)<\delta}} \abs{Q_n(z,z_0)} \leq (n+1)
C_\veps^2 e^{2\veps\abs{n}}
\end{equation}

The other reason regularity enters has to do with the density of zeros. If $x_j^{(n)}$ are the zeros of
$p_n(x,d\mu)$, we define the zero counting measure, $d\nu_n$, to be the probability measure that gives
weight to $n^{-1}$ to each $x_j^{(n)}$. For the following, see \cite{StT,EqMC}:

\begin{theorem}\lb{T12.2} Let $\fre\subset\bbR$ be compact and let $\mu$ be a regular measure for $\fre$.
Then
\begin{equation} \lb{12.5}
d\nu_n\to d\rho_\fre
\end{equation}
the equilibrium measure for $\fre$.
\end{theorem}

In \eqref{12.5}, the convergence is weak.

\section{Weak Limits} \lb{s13}

A major theme in the remainder of this review is pointwise asymptotics of $\f{1}{n+1}
K_n(x,y;d\mu)$ and its diagonal. Therefore, it is interesting that one can say something
about $\f{1}{n+1} K_n(x,x;d\mu)\, d\mu(x)$ without pointwise asymptotics. Notice that
\begin{equation} \lb{13.1}
d\mu_n(x) \equiv \f{1}{n+1}\, K_n(x,x;d\mu)\, d\mu(x)
\end{equation}
is a probability measure. Recall the density of zeros, $\nu_n$, defined after \eqref{12.4}.

\begin{theorem}\lb{T13.1} Let $\mu$ have compact support. Let $\nu_n$ be the density of zeros
and $\mu_n$ given by \eqref{13.1}. Then for any $\ell=0,1,2,\dots$,
\begin{equation} \lb{13.2}
\biggl| \int x^\ell\, d\nu_{n+1} - \int x^\ell\, d\mu_n\biggr| \to 0
\end{equation}
In particular, $d\mu_{n(j)}$ and $d\nu_{n(j)+1}$ have the same weak limits for any subsequence $n(j)$.
\end{theorem}

\begin{proof} By Theorem~\ref{T6.2}, the zeros of $P_{n+1}$ are eigenvalues of $\pi_n M_x \pi_n$, so
\begin{equation} \lb{13.3}
\int x^\ell \,d\nu_{n+1} = \f{1}{n+1}\, \tr((\pi_n M_x \pi_n)^\ell)
\end{equation}

On the other hand, since $\{p_j\}_{j=0}^n$ is a basis for $\ran(\pi_n)$,
\begin{align}
\int x^\ell\, d\mu_n &= \f{1}{n+1} \sum_{j=0}^n \int x^\ell \abs{p_j(x)}^2\, d\mu(x) \notag \\
&= \f{1}{n+1}\, \tr (\pi_n M_x^\ell \pi_n) \lb{13.4}
\end{align}

It is easy to see that $(\pi_n M_x \pi_n)^\ell - \pi_n M_x^\ell \pi_n$ is rank at most $\ell$, so
\[
\text{LHS of \eqref{13.2}} \leq \f{\ell}{n+1}\, \Norm{M_x}^\ell
\]
goes to $0$ as $n\to\infty$ for $\ell$ fixed.
\end{proof}

\begin{remark} This theorem is due to Simon \cite{Weak-cd} although the basic fact goes back to
Avron--Simon \cite{S149}.
\end{remark}

See Simon \cite{Weak-cd} for an interesting application to comparison theorems for limits of density of states.
We immediately have:

\begin{corollary} \lb{C13.2} Suppose that
\begin{equation} \lb{13.5}
d\mu = w(x)\, dx + d\mu_\s
\end{equation}
with $d\mu_\s$ Lebesgue singular, and on some open interval $I\subset\fre=\sigma_\ess (d\mu)$ we have
$d\nu_n\to d\nu_\infty$ and
\begin{equation} \lb{13.6}
d\nu_\infty \restriction I = \nu_\infty(x)\, dx
\end{equation}
and suppose that uniformly on $I$,
\begin{equation} \lb{13.7}
\lim\, \f{1}{n} K_n(x,x) = g(x)
\end{equation}
and $w(x)\neq 0$ on $I$. Then
\[
g(x) = \f{\nu_\infty(x)}{w(x)}
\]
\end{corollary}

\begin{proof} The theorem implies $d\nu_\infty\restriction I =w(x)g(x)$.
\end{proof}

Thus, in the regular case, we expect that ``usually"
\begin{equation} \lb{13.8}
\f{1}{n}\, K_n(x,x) \to \f{\rho_\fre(x)}{w(x)}
\end{equation}
This is what we explore in much of the rest of this paper.

\section{Variational Principle: M\'at\'e--Nevai Upper Bounds} \lb{s14}

The Cotes numbers, $\lambda_n (z_0)$, are given by \eqref{9.6}, so upper bounds on $\lambda_n (z_0)$
mean lower bounds on diagonal CD kernels and there is a confusion of ``upper bounds'' and ``lower bounds.''
We will present here some very general estimates that come from the use of trial functions in \eqref{9.6}
so they are called M\'at\'e--Nevai upper bounds (after \cite{MNT87b}), although we will write them as lower
bounds on $K_n$. One advantage is their great generality.

\begin{definition} Let $d\mu$ be a measure on $\bbR$ of the form
\begin{equation} \lb{14.1}
d\mu = w(x)\, dx + d\mu_\s
\end{equation}
where $d\mu_\s$ is singular with respect to Lebesgue measure. We call $x_0$ a Lebesgue point of $\mu$
if and only if
\begin{gather}
\f{n}{2}\, \mu_\s \biggl(\biggl[ x_0 - \f{1}{n}, x_0 + \f{1}{n}\biggr]\biggr)\to 0 \lb{14.2} \\
\f{n}{2} \int_{x_0-\f{1}{n}}^{x_0 + \f{1}{n}} \abs{w(x)-w_0(x_0)}\, dx \to 0 \lb{14.3}
\end{gather}
\end{definition}

It is a fundamental fact of harmonic analysis (\cite{Rudin}) that for any $\mu$ Lebesgue-a.e., $x_0$ in
$\bbR$ is a Lebesgue point for $\mu$. Here is the most general version of the MN upper bound:

\begin{theorem}\lb{T14.1} Let $\fre\subset\bbR$ be an arbitrary compact set which is regular for the Dirichlet
problem. Let $I\subset\fre$ be a closed interval. Let $d\mu$ be a measure with compact support in
$\bbR$ with $\sigma_\ess (d\mu)\subset\fre$. Then for any Lebesgue point, $x$ in $I$,
\begin{equation} \lb{14.4}
\liminf_{n\to\infty} \f{1}{n}\, K_n(x,x) \geq \f{\rho_\fre(x)}{w(x)}
\end{equation}
where $d\rho_\fre\restriction I=\rho_\fre(x)\, dx$. If $w$ is continuous on $I$ {\rm{(}}including at
the endpoints as a function in a neighborhood of $I${\rm{)}} and nonvanishing, then \eqref{14.4} holds
uniformly on $I$. If $x_n\to x\in I$ and $A=\sup_n n\abs{x_n-x} <\infty$ and $x$ is a Lebesgue, then
\eqref{14.4} holds with $K_n(x,x)$ replaced by $K_n(x_n,x_n)$. If $w$ is continuous and nonvanishing on $I$,
then this extended convergence is uniform in $x\in I$ and $x_n$'s with $A\leq A_0 <\infty$.
\end{theorem}

\begin{remarks} 1. If $I\subset\fre$ is a nontrivial interval, the measure $d\rho_\fre\restriction I$ is
purely absolutely continuous (see, e.g., \cite{EqMC,Rice}).

\smallskip
2. For OPUC, this is a result of M\'at\'e--Nevai \cite{MN}. The translation to OPRL on $[-1,1]$ is
explicit in M\'at\'e--Nevai--Totik \cite{MNT91}. The extension to general sets via polynomial mapping
and approximation (see Section~\ref{s18}) is due to Totik \cite{Tot}. These papers also require a local
Szeg\H{o} condition, but that is only needed for lower bounds on $\lambda_n$ (see Section~\ref{s17}).
They also don't state the $x_n\to x_\infty$ result, which is a refinement introduced by Lubinsky
\cite{Lub} who implemented it in certain $[-1,1]$ cases.

\smallskip
3. An alternate approach for Totik's polynomial mapping is to use trial functions based on Jost--Floquet
solutions for periodic problems; see Section~\ref{s19} (and also \cite{2exts,Rice}).
\end{remarks}

One can combine \eqref{14.4} with weak convergence and regularity to get

\begin{theorem}[Simon \cite{Weak-cd}]\lb{T14.2} Let $\fre\subset\bbR$ be an arbitrary compact set, regular
for the Dirichlet problem. Let $d\mu$ be a measure with compact support in $\bbR$ with $\sigma_\ess (d\mu)
= \fre$ and with $d\mu$ regular for $\fre$. Let $I\subset\fre$ be an interval so $w(x) >0$ a.e.\ on $I$. Then
\begin{alignat}{2}
&\text{\rm{(i)}} \qquad && \int_I \biggl| \f{1}{n}\, K_n(x,x)\, w(x)-\rho_\fre(x)\biggr| dx \to 0 \lb{14.5} \\
&\text{\rm{(ii)}} \qquad && \int_I \f{1}{n}\, K_n(x,x)\, d\mu_\s (x) \to 0 \lb{14.6}
\end{alignat}
\end{theorem}

\begin{proof} By Theorems~\ref{T12.2} and \ref{T13.1},
\begin{equation} \lb{14.7}
\f{1}{n}\, K_n(x,x)\, d\mu \to d\rho_\fre
\end{equation}
Let $\nu_1$ be a limit point of $\f{1}{n} K_n(x,x)\, d\mu_\s$ and
\begin{equation} \lb{14.8}
d\nu_2 = d\rho_\fre - d\nu_1
\end{equation}

If $f\geq 0$, by Fatou's lemma and \eqref{14.4},
\begin{equation} \lb{14.9}
\int_I f\, d\nu_2 \geq \int_I \rho_\fre(x) f(x)\, dx
\end{equation}
that is, $d\nu_2 \restriction I \geq \rho_\fre(x)\, dx \restriction I$. By \eqref{14.8}, $d\nu_2
\restriction I\leq \rho_\fre(x)\, dx$. It follows $d\nu_1 \restriction I$ is $0$ and $d\nu_2
\restriction I = d\rho_\fre \restriction I$.

By compactness, $\f{1}{n} K_n(x,x)\, d\mu_\s \restriction I\to 0$ weakly, implying \eqref{14.6}. By a
simple argument \cite{Weak-cd}, weak convergence of $\f{1}{n} K_n(x,x) w(x)\, dx \to \rho_\fre(x)\, dx$
and \eqref{14.4} imply \eqref{14.5}.
\end{proof}

\section{Criteria for A.C.\ Spectrum} \lb{s15}

Define
\begin{equation} \lb{15.1}
N=\biggl\{x\in\bbR\biggm| \liminf \, \f{1}{n}\, K_n(x,x)<\infty\biggr\}
\end{equation}
so that
\begin{equation} \lb{15.2}
\bbR\setminus N =\biggl\{x\in\bbR\biggm| \lim \,\f{1}{n}\, K_n(x,x)=\infty\biggr\}
\end{equation}

Theorem~\ref{T14.1} implies

\begin{theorem}\lb{T15.1} Let $\fre\subset\bbR$ be an arbitrary compact set and $d\mu=w(x)\, dx +
d\mu_\s$ a measure with $\sigma(\mu)=\fre$. Let $\Sigma_\ac=\{x\mid w(x) >0\}$. Then $N\setminus
\Sigma_\ac$ has Lebesgue measure zero.
\end{theorem}

\begin{proof} If $x_0\in\bbR\setminus\Sigma_\ac$ and is a Lebesgue point of $\mu$, then $w(x_0)=0$ and,
by Theorem~\ref{T14.1}, $x_0\in\bbR\setminus N$. Thus,
\[
(\bbR\setminus\Sigma_\ac)\setminus (\bbR\setminus N) = N\setminus\Sigma_\ac
\]
has Lebesgue measure zero.
\end{proof}

\begin{remark}  This is a direct but not explicit consequence of the M\'at\'e--Nevai ideas
\cite{MN}. Without knowing of this work, Theorem~\ref{T15.1} was rediscovered with a very
different proof by Last--Simon \cite{LastS}.
\end{remark}

On the other hand, following Last--Simon \cite{LastS}, we note that Fatou's lemma and
\begin{equation} \lb{15.3}
\int \f{1}{n}\, K_n(x,x)\, d\mu(x) =1
\end{equation}
implies
\begin{equation} \lb{15.4}
\int \liminf\, \f{1}{n}\, K_n(x,x)\, d\mu(x) \leq 1
\end{equation}
so

\begin{theorem}[\cite{LastS}]\lb{T15.2} $\Sigma_\ac\setminus N$ has Lebesgue measure zero.
\end{theorem}

Thus, up to sets of measure zero, $\Sigma_\ac =N$. What is interesting is that this holds,
for example, when $\fre$ is a positive measure Cantor set as occurs for the almost Mathieu
operator $(a_n\equiv 1$, $b_n =\lambda \cos(\pi\alpha n + \theta)$, $\abs{\lambda} <2$,
$\lambda\neq 0$, $\alpha$ irrational). This operator has been heavily studied; see
Last \cite{Last05}.

\section{Variational Principle: Nevai Trial Polynomial} \lb{s16}

A basic idea is that if $d\mu_1$ and $d\mu_2$ look alike near $x_0$, there is a good chance that
$K_n(x_0,x_0;d\mu_1)$ and $K_n(x_0,x_0;d\mu_2)$ are similar for $n$ large. The expectation \eqref{13.8}
says they better have the same support (and be regular for that support), but this is a reasonable
guess.

It is natural to try trial polynomials minimizing $\lambda_n (x_0,d\mu_1)$ in the Christoffel
variational principle for $\lambda_n (x_0,d\mu_2)$, but Example~11.3 shows this will not work
in general. If $d\mu_1$ has a strong zero near some other $x_1$, the trial polynomial for $d\mu_1$
may be large near $x_1$ and be problematical for $d\mu_2$ if it does not have a zero there. Nevai
\cite{Nev79} had the idea of using a localizing factor to overcome this.

Suppose $\fre\subset\bbR$, a compact set which, for now, we suppose contains $\sigma(d\mu_1)$ and
$\sigma(d\mu_2)$. Pick $A=\diam(\fre)$ and consider (with $[\dott]\equiv$ integral part)
\begin{equation} \lb{16.1}
\biggl( 1-\f{(x-x_0)^2}{A^2}\biggr)^{[\veps n]} \equiv N_{2[\veps n]}(x)
\end{equation}
Then for any $\delta$,
\begin{equation} \lb{16.2}
\sup_{\substack{ \abs{x-x_0} > \delta \\ x\in\fre}} N_{2[\veps n]}(x) \leq e^{-c(\delta,\veps)n}
\end{equation}
so if $Q_{n-2[\veps n]}(x)$ is the minimizer for $\mu_1$ and $\fre$ is regular for the Dirichlet
problem and $\mu_1$ is regular for $\fre$, then the Nevai trial function
\[
N_{2[\veps n]}(x) Q_{n-2[\veps n]}(x)
\]
will be exponentially small away from $x_0$.

For this to work to compare $\lambda_n (x_0,d\mu_1)$ and $\lambda (x_0,d\mu_2)$, we need two
additional properties of $\lambda_n (x_0,d\mu_1)$:
\begin{SL}
\item[(a)] $\lambda_n (x_0,d\mu_1)\geq C_\veps e^{-\veps n}$ for each $\veps <0$. This is needed for the
exponential contributions away from $x_0$ not to matter.
\item[(b)]
\[
\lim_{\veps\downarrow 0}\, \limsup_{n\to\infty}\, \f{\lambda_n(x_0,d\mu_1)}
{\lambda_{n-2[\veps n]}(x,d\mu_1)}=1
\]
so that the change from $Q_n$ to $Q_{n-2[\veps n]}$ does not matter.
\end{SL}

\smallskip
Notice that both (a) and (b) hold if
\begin{equation} \lb{16.3}
\lim_{n\to\infty}\, n\lambda_n(x_0,d\mu)=c>0
\end{equation}

If one only has $\fre=\sigma_\ess (d\mu_2)$, one can use explicit zeros in the trial polynomials to mask
the eigenvalues outside $\fre$.

For details of using Nevai trial functions, see \cite{2exts,Rice}. Below we will just refer to using
Nevai trial functions.

\section{Variational Principle: M\'at\'e--Nevai--Totik Lower Bound} \lb{s17}

In \cite{MNT91}, M\'at\'e--Nevai--Totik proved:

\begin{theorem} \lb{T17.1} Let $d\mu$ be a measure on $\partial\bbD$
\begin{equation} \lb{17.1}
d\mu = \f{w(\theta)}{2\pi}\, d\theta + d\mu_\s
\end{equation}
which obeys the Szeg\H{o} condition
\begin{equation} \lb{17.2}
\int \log(w(\theta))\, \f{d\theta}{2\pi} > -\infty
\end{equation}
Then for a.e.\ $\theta_\infty\in\partial\bbD$,
\begin{equation} \lb{17.3}
\liminf\, n\lambda_n (\theta_\infty) \geq w(\theta_\infty)
\end{equation}
This remains true if $\lambda_n (\theta_\infty)$ is replaced by $\lambda_n (\theta_n)$
with $\theta_n\to\theta_\infty$ obeying $\sup n\abs{\theta_n-\theta_\infty}<\infty$.
\end{theorem}

\begin{remarks} 1. The proof in \cite{MNT91} is clever but involved (\cite{Rice} has an
exposition); it would be good to find a simpler proof.

\smallskip
2. \cite{MNT91} only has the result $\theta_n =\theta_\infty$. The general $\theta_n$ result is
due to Findley \cite{Find}.

\smallskip
3. The $\theta_\infty$ for which this is proven have to be Lebesgue points for $d\mu$ as well as
Lebesgue points for $\log(w)$ and for its conjugate function.

\smallskip
4. As usual, if $I$ is an interval with $w$ continuous and nonvanishing, and $\mu_\s(I)=0$,
\eqref{17.3} holds uniformly if $\theta_\infty\in I$.
\end{remarks}

By combining this lower bound with the M\'at\'e--Nevai upper bound, we get the result of
M\'at\'e--Nevai--Totik \cite{MNT91}:

\begin{theorem}\lb{T17.2} Under the hypothesis of Theorem~\ref{T17.1}, for a.e.\ $\theta_\infty\in
\partial\bbD$,
\begin{equation} \lb{17.4}
\lim_{n\to\infty}\, n\lambda_n(\theta_\infty) = w(\theta_\infty)
\end{equation}
This remains true if $\lambda_n (\theta_\infty)$ is replaced by $\lambda_n (\theta_n)$ with
$\theta_n \to \theta_\infty$ obeying $\sup n\abs{\theta_n-\theta_\infty}<\infty$. If $I$ is an
interval with $w$ continuous on $I$ and $\mu_\s(I)=0$, then these results hold uniformly in $I$.
\end{theorem}

\begin{remark} It is possible (see remarks in Section~4.6 of \cite{NevFr}) that \eqref{17.4}
holds if a Szeg\H{o} condition is replaced by $w(\theta) >0$ for a.e.\ $\theta$. Indeed,
under that hypothesis, Simon \cite{Weak-cd} proved that
\[
\int_0^{2\pi} \abs{w(\theta)(n\lambda_n(\theta))^{-1} -1}\, \f{d\theta}{2\pi} \to 0
\]
\end{remark}

There have been significant extensions of Theorem~\ref{T17.2} to OPRL on fairly general sets:
\begin{SL}
\item[1.] \cite{MNT91} used the idea of Nevai trial functions (Section~\ref{s16}) to prove the
Szeg\H{o} condition could be replaced by regularity plus a local Szeg\H{o} condition.

\item[2.] \cite{MNT91} used the Szeg\H{o} mapping to get a result for $[-1,1]$.

\item[3.] Using polynomial mappings (see Section~\ref{s18}) plus approximation, Totik \cite{Tot}
proved a general result (see below); one can replace polynomial mappings by Floquet--Jost solutions
(see Section~\ref{s19}) in the case of continuous weights on an interval (see \cite{2exts}).
\end{SL}

\smallskip
Here is Totik's general result (extended from $\sigma(d\mu)\subset\fre$ to $\sigma_\ess (d\mu)
\subset\fre$):

\begin{theorem}[Totik \cite{Tot,Tot-prep}]\lb{T17.3} Let $\fre$ be a compact subset of $\bbR$.
Let $I\subset\fre$ be an interval. Let $d\mu$ have $\sigma_\ess
(d\mu) = \fre$ be regular for $\fre$ with
\begin{equation} \lb{17.5}
\int_I \log(w)\, dx > -\infty
\end{equation}
Then for a.e.\ $x_\infty\in I$,
\begin{equation} \lb{17.6}
\lim_{n\to\infty}\, \f{1}{n}\, K_n(x_\infty, x_\infty) = \f{\rho_\fre(x_\infty)}{w(x_\infty)}
\end{equation}
The same limit holds for $\f{1}{n} K_n(x_n,x_n)$ if $\sup_n n\abs{x_n-x_\infty} <\infty$. If $\mu_\s(I)
=\emptyset$ and $w$ is continuous and nonvanishing on $I$, then those limits are uniform on $x_\infty
\in I$ and on all $x_n$'s with $\sup_n n\abs{x_n-x_\infty}\leq A$ {\rm{(}}uniform for each fixed
$A${\rm{)}}.
\end{theorem}

\begin{remarks} 1. Totik \cite{Tot-Chr-prep} recently proved asymptotic results for suitable CD kernels
for OPs which are neither OPUC nor OPRL.

\smallskip
2. The extension to general compact $\fre$ without an assumption of regularity for the Dirichlet
problem is in \cite{Tot-prep}.
\end{remarks}

\section{Variational Principle: Polynomial Maps} \lb{s18}

In passing from $[-1,1]$ to fairly general sets, one uses a three-step process. A finite gap
set is an $\fre$ of the form
\begin{equation} \lb{18.1}
\fre=[\alpha_1,\beta_1] \cup [\alpha_2, \beta_2] \cup \cdots \cup [\alpha_{\ell+1}, \beta_{\ell+1}]
\end{equation}
where
\begin{equation} \lb{18.2}
\alpha_1 < \beta_1 < \alpha_2 < \beta_2 < \cdots < \alpha_{\ell+1} < \beta_{\ell+1}
\end{equation}

$\calE_f$ will denote the family of finite gap sets. We write $\fre=\fre_1 \cup \cdots \cup \fre_{\ell+1}$
in this case with the $\fre_j$ closed disjoint intervals. $\calE_p$ will denote the set of what we called
periodic finite gap sets in Section~\ref{s10}---ones where each $\fre_j$ has rational harmonic measure.
Here are the three steps:
\begin{SL}
\item[(1)] Extend to $\fre\in\calE_p$ using the methods discussed briefly below.
\item[(2)] Prove that given any $\fre\in\calE_f$, there is $\fre^{(n)}\in\calE_p$, each with the same
number of bands so $\fre_j\subset\fre_j^{(n)}\subset\fre_j^{(n-1)}$ and $\cap_n \fre_j^{(n)} = \fre_j$.
This is a result proven independently by Bogatyr\"ev \cite{Bog}, Peherstorfer \cite{Peh}, and Totik
\cite{Tot-acta}; see \cite{Rice} for a presentation of Totik's method.

\item[(3)] Note that for any compact $\fre$, if $\fre^{(m)}=\{x\mid \dist(x,\fre)\leq \f{1}{m}\}$, then
$\fre^{(m)}$ is a finite gap set and $\fre = \cap_m \fre^{(m)}$.
\end{SL}

\smallskip
Step (1) is the subtle step in extending theorems: Given the Bogatyr\"ev--Peherstorfer--Totik theorem,
the extensions are simple approximation.

The key to $\fre\in\calE_p$ is that there is a polynomial $\ti\Delta\colon\bbC\to\bbC$, so $\ti\Delta^{-1}
([-1,1])=\fre$ and so that $\fre_j$ is a finite union of intervals $\ti\fre_k$ with disjoint interiors
so that $\ti\Delta$ is a bijection from each $\ti\fre_k$ to $[-1,1]$. That this could be useful
was noted initially by Geronimo--Van Assche \cite{GVA}. Totik showed how to prove Theorem~\ref{T17.3}
for $\fre\in\calE_p$ from the results for $[-1,1]$ using this polynomial mapping.

For spectral theorists, the polynomial $\ti\Delta=\f12\Delta$ where $\Delta$ is the discriminant for the
associated periodic problem (see \cite{Hoch7,Last92,vMoer,Toda,Rice}). There is a direct construction of
$\ti\Delta$ by Aptekarev \cite{Apt} and Peherstorfer \cite{Pe93,Peh,Pe03}.

\section{Floquet--Jost Solutions for Periodic Jacobi Matrices} \lb{s19}

As we saw in Section~\ref{s16}, models with appropriate behavior are useful input for comparison
theorems. Periodic Jacobi matrices have OPs for which one can study the CD kernel and its asymptotics.
The two main results concern diagonal and just off-diagonal behavior:

\begin{theorem}\lb{T19.1} Let $\mu$ be the spectral measure associated to a periodic Jacobi matrix with
essential spectrum, $\fre$, a finite gap set. Let $d\mu=w(x)\, dx$ on $\fre$ {\rm{(}}there can also be
up to one eigenvalue in each gap{\rm{)}}. Then uniformly for $x$ in compact subsets of $\fre^\intt$,
\begin{equation} \lb{19.1}
\f{1}{n}\, K_n(x,x) \to \f{\rho_\fre(x)}{w(x)}
\end{equation}
and uniformly for such $x$ and $a,b$ in $\bbR$ with $\abs{a}\leq A$, $\abs{b}\leq B$,
\begin{equation} \lb{19.2}
\f{K_n(x+\f{a}{n}, x+\f{b}{n})}{K_n(x,x)} \to \f{\sin(\pi\rho_\fre(x) (b-a))}{\pi\rho_\fre(x)(b-a)}
\end{equation}
\end{theorem}

\begin{remarks} 1. \eqref{19.2} is often called bulk universality. On bounded intervals, it goes back
to random matrix theory. The best results using Riemann--Hilbert methods for OPs is due to
Kuijlaars--Vanlessen \cite{KV}.  A different behavior is expected at the edge of the spectrum---we
will not discuss this in detail, but see Lubinsky \cite{Lubppt}.

\smallskip
2. For $[-1,1]$, Lubinsky \cite{Lub} used Legendre polynomials as his model. The references for the
proofs here are Simon \cite{2exts, Rice}.
\end{remarks}

The key to the proof of Theorem~\ref{T19.1} is to use Floquet--Jost solutions, that is, solutions of
\begin{equation} \lb{19.3}
a_n u_{n+1} + b_n u_n + a_{n-1} u_{n-1} = xu_n
\end{equation}
for $n\in\bbZ$ where $\{a_n, b_n\}$ are extended periodically to all of $\bbZ$. These solutions obey
\begin{equation} \lb{19.4}
u_{n+p} = e^{i\theta(x)} u_n
\end{equation}
For $x\in\fre^\intt$, $u_n$ and $\bar u_n$ are linearly independent, and so one can write $p_{\bddot -1}$
in terms of $u_\bddot$ and $\bar u_\bddot$. Using
\begin{equation} \lb{19.5}
\rho_\fre(x) = \f{1}{p\pi}\, \biggl| \f{d\theta}{dx}\biggr|
\end{equation}
one can prove \eqref{19.1} and \eqref{19.2}. The details are in \cite{2exts, Rice}.

\section{Lubinsky's Inequality and Bulk Universality} \lb{s20}

Lubinsky \cite{Lub} found a powerful tool for going from diagonal control of the CD kernel to slightly
off-diagonal control---a simple inequality.

\begin{theorem}\lb{T20.1} Let $\mu\leq\mu^*$ and let $K_n,K_n^*$ be their CD kernels. Then for any
$z,\zeta$,
\begin{equation} \lb{20.1}
\abs{K_n(z,\zeta) - K_n^*(z,\zeta)}^2 \leq K_n(z,z) [K_n(\zeta,\zeta)-K_n^*(\zeta,\zeta)]
\end{equation}
\end{theorem}

\begin{remark} Recall (Theorem~\ref{T9.3}) that $K_n(\zeta,\zeta)\geq K_n^*(\zeta,\zeta)$.
\end{remark}

\begin{proof} Since $K_n-K_n^*$ is a polynomial $\bar z$ of degree $n$:
\begin{equation} \lb{20.2}
K_n(z,\zeta) - K_n^*(z,\zeta) = \int K_n(z,w) [K_n (w,\zeta) - K_n^*(w,\zeta)]\, d\mu(w)
\end{equation}
By the reproducing kernel formula \eqref{1.12}, we get \eqref{20.1} from the Schwarz inequality
if we show
\begin{equation} \lb{20.3}
\int \abs{K_n(w,\zeta) - K_n^*(w,\zeta)}^2\, d\mu(w) \leq K_n(\zeta,\zeta) - K_n^*(\zeta,\zeta)
\end{equation}

Expanding the square, the $K_n^2$ term is $K_n(\zeta,\zeta)$ by \eqref{1.12} and the $K_n K_n^*$
cross term is $-2K_n^*(\zeta,\zeta)$ by the reproducing property of $K_n$ for $d\mu$ integrals.
Thus, \eqref{20.3} is equivalent to
\begin{equation} \lb{20.4}
\int \abs{K_n^*(w,\zeta)}^2\, d\mu(w) \leq K_n^*(\zeta,\zeta)
\end{equation}
This in turn follows from $\mu\leq\mu^*$ and \eqref{1.12} for $\mu^*$!
\end{proof}

This result lets one go from diagonal control on measures to off-diagonal. Given any pair of measures,
$\mu$ and $\nu$, there is a unique measure $\mu\vee\nu$ which is their least upper bound (see, e.g.,
Doob \cite{Doob}). It is known (see \cite{EqMC}) that if $\mu,\nu$ are regular for the same set,
so is $\mu\vee\nu$. \eqref{20.1} immediately implies that (go from $\mu$ to $\mu^*$ and then
$\mu^*$ to $\nu$):

\begin{corollary}\lb{C20.2} Let $\mu,\nu$ be two measures and $\mu^* = \mu\vee \nu$. Suppose for some
$z_n\to z_\infty$, $w_n\to z_\infty$, we have for $\eta=\mu,\nu,\mu^*$ that
\[
\lim_{n\to\infty}\, \f{K_n(z_n,z_n;\eta)}{K_n(z_\infty,z_\infty;\eta)} =
\lim_{n\to\infty}\, \f{K_n(w_n,w_n;\eta)}{K_n(z_\infty,z_\infty;\eta)} =1
\]
and that
\[
\lim_{n\to\infty}\, \f{K_n(z_\infty,z_\infty;\mu)}{K_n(z_\infty,z_\infty;\mu^*)} =
\lim_{n\to\infty}\, \f{K_n(z_\infty,z_\infty;\nu)}{K_n(z_\infty,z_\infty;\mu^*)} = 1
\]
Then
\begin{equation} \lb{20.5}
\lim_{n\to\infty}\, \f{K_n(z_n,w_n;\mu)}{K_n(z_n,w_n;\nu)} = 1
\end{equation}
\end{corollary}

\begin{remark} It is for use with $x_n = x_\infty + \f{a}{n}$ or $x_\infty + \f{a}{\rho_n n}$
that we added $x_n\to x_\infty$ to the various diagonal kernel results. This ``wiggle'' in $x_\infty$
was introduced by Lubinsky \cite{Lub}, so we dub it the ``Lubinsky wiggle.''
\end{remark}

Given Totik's theorem (Theorem~\ref{T17.3}) and bulk universality for suitable models, one
thus gets:

\begin{theorem}\lb{T20.3} Under the hypotheses of Theorem~\ref{T17.3}, for a.e.\ $x_\infty$ in
$I$, we have uniformly for $\abs{a},\abs{b} <A$,
\[
\lim_{n\to\infty}\, \f{K_n(x_\infty + \f{a}{n},x_\infty + \f{b}{n})}{K_n(x_\infty,x_\infty)} =
\f{\sin(\pi\rho_\fre (x_\infty)(b-a))}{\pi\rho_\fre (x_\infty)(b-a)}
\]
\end{theorem}

\begin{remarks} 1. For $\fre=[-1,1]$, the result and method are from Lubinsky \cite{Lub}.

\smallskip
2. For continuous weights, this is in Simon \cite{2exts} and Totik \cite{Tot-prep}, and
for general weights, in Totik \cite{Tot-prep}.
\end{remarks}

\section{Derivatives of CD Kernels} \lb{s21}

The ideas in this section come from a paper in preparation with Avila and Last \cite{ALSinprep}.
Variation of parameters is a standard technique in ODE theory and used as an especially
powerful tool in spectral theory by Gilbert--Pearson \cite{GP} and in Jacobi matrix spectral
theory by Khan--Pearson \cite{KPe}. It was then developed by Jitomirskaya--Last \cite{JL0,JL1,JL2}
and Killip--Kiselev--Last \cite{KKL}, from which we take Proposition~\ref{P21.1}.

\begin{proposition}\lb{P21.1} For any $x,x_0$, we have
\begin{equation} \lb{21.2}
p_n(x)-p_n(x_0) = (x-x_0) \sum_{m=0}^{n-1} (p_n(x_0) q_m(x_0)-p_m(x_0) q_n(x_0)) p_m(x)
\end{equation}
In particular,
\begin{equation} \lb{21.3}
p'_n (x_0) =\sum_{m=0}^{n-1} (p_n(x_0) q_m(x_0) - p_m(x_0) q_n(x_0)) p_m(x_0)
\end{equation}
\end{proposition}

Here $q_n$ are the second kind polynomials defined in Section~\ref{s8}. For \eqref{21.2},
see \cite{JL0,JL1,JL2,KKL}. This immediately implies:

\begin{corollary}[Avila--Last--Simon \cite{ALSinprep}] \lb{C21.2}
\begin{equation} \lb{21.4}
\begin{split}
&\left. \f{d}{da}\,  \f{1}{n}\, K_n\biggl(x_0 + \f{a}{n}, x_0 + \f{a}{n}\biggr) \right|_{a=0} \\
&\quad = \f{2}{n^2} \sum_{j=0}^n \biggl[ p_j(x_0)^2 \biggl(\, \sum_{k=0}^j p_k(x_0) q_k(x_0)\biggr)
-q_j (x_0) p_j(x_0) \biggl(\, \sum_{k=0}^j p_k(x_0)^2 \biggr) \biggr]
\end{split}
\end{equation}
\end{corollary}

This formula gives an indication of why (as we see in the next section is important) $\lim \f{1}{n}
K_n(x_0 + \f{a}{n}, x_0 + \f{a}{n})$ has a chance to be independent of $a$ if one notes the
following fact:

\begin{lemma}\lb{L21.3} If $\{\alpha_n\}_{n=1}^\infty$ and $\{\beta_n\}_{n=1}^\infty$ are sequences
so that $\lim\f{1}{N}\sum_{n=1}^N \alpha_n =A$ and $\lim\f{1}{N} \sum_{n=1}^N \beta_n= B$ exist
and $\sup_N [\f{1}{N} \sum_{n=1}^N \abs{\alpha_n} + \abs{\beta_n}]<\infty$, then
\begin{equation} \lb{21.5}
\f{1}{N^2} \sum_{j=1}^N \biggl[\biggl( \alpha_j \sum_{k=1}^j \beta_k\biggr) -
\biggl(\beta_j \sum_{k=1}^j \alpha_k\biggr)\biggr] \to 0
\end{equation}
\end{lemma}

This is because
\[
\f{1}{N^2}\sum_{j=1}^N \alpha_j \sum_{k=1}^j \beta_k \to \f12\, AB
\]
Setting $\alpha_j =p_j(x_0)^2$ and $\beta_j = p_j(x_0) q_j(x_0)$, one can hope to use \eqref{21.5}
to prove the right side of \eqref{21.4} goes to zero.

\section{Lubinsky's Second Approach} \lb{s22}

Lubinsky revolutionized the study of universality in \cite{Lub}, introducing the approach we described
in Section~\ref{s20}. While Totik \cite{Tot-prep} and Simon \cite{2exts} used those ideas to extend
beyond the case of $\fre=[-1,1]$ treated in \cite{Lub}, Lubinsky developed a totally different approach
\cite{Lub-jdam} to go beyond \cite{Lub}. That approach, as abstracted in Avila--Last--Simon \cite{ALSinprep},
is discussed in this section. Here is an abstract theorem:

\begin{theorem}\lb{T22.1} Let $d\mu$ be a measure of compact support on $\bbR$. Let $x_0$ be a Lebesgue
point for $\mu$ and suppose that
\begin{SL}
\item[{\rm{(i)}}] For any $\veps$, there is a $C_\veps$ so that for any $R$, we have an $N(\veps,R)$ so
that for $n\geq N(\veps,R)$,
\begin{equation} \lb{22.1}
\f{1}{n}\, K_n \biggl( x_0 + \f{z}{n}, x_0 + \f{z}{n}\biggr) \leq C_\veps e^{\veps\abs{z}^2}
\end{equation}
for all $z\in\bbC$ with $\abs{z}<R$.

\item[{\rm{(ii)}}] Uniformly for real $a$'s in compact subsets of $\bbR$,
\begin{equation} \lb{22.2}
\lim_{n\to\infty}\, \f{K_n(x_0 + \f{a}{n}\,, x_0 + \f{a}{n})}{K_n(x_0,x_0)} =1
\end{equation}
Let
\begin{equation} \lb{22.3}
\rho_n = \f{w(x_0)}{n}\, K_n (x_0,x_0)
\end{equation}
\end{SL}

Then uniformly for $z,w$ in compact subsets of $\bbC$,
\begin{equation} \lb{22.4}
\lim_{n\to\infty}\, \f{K_n(x_0 + \f{z}{n\rho_n}, x_0 + \f{w}{n\rho_n})}{K_n(x_0,x_0)} =
\f{\sin(\pi (\bar z-w))}{\pi (\bar z -w)}
\end{equation}
\end{theorem}

\begin{remarks} 1. If $\rho_n\to\rho_\fre(x_0)$, the density of the equilibrium measure, then \eqref{22.4}
is the same as \eqref{19.2}. In every case where Theorem~\ref{T22.1} has been proven to be applicable (see
below), $\rho_n\to\rho_\fre(x_0)$. But one of the interesting aspects of this is that it might apply in
cases where $\rho_n$ does not have a limit. For an example with a.c.\ spectrum but where the density of
zeros has multiple limits, see Example~5.8 of \cite{EqMC}.

\smallskip
2. Lubinsky \cite{Lub-jdam} worked in a situation (namely, $x_0$ in an interval $I$ with $w(x_0) \geq
c > 0$ on $I$) where \eqref{22.1} holds in the stronger form $Ce^{D\abs{z}}$ (no square on $\abs{z}$)
and used arguments that rely on this. Avila--Last--Simon \cite{ALSinprep} found the result stated here; the
methods seem incapable of working with \eqref{22.1} for a fixed $\veps$ rather than all $\veps$
(see Remark~1 after Theorem~\ref{T22.2}).
\end{remarks}

Let us sketch the main ideas in the proof of Theorem~\ref{T22.1}:

(1) \ By \eqref{15.1},
\begin{equation} \lb{22.5}
\liminf \f{1}{n}\, K_n(x_0,x_0) > 0
\end{equation}

\smallskip
(2) \ By the Schwarz inequality \eqref{1.13a}, \eqref{22.1}, and \eqref{22.5}, and by the compactness of
normal families, we can find subsequences $n(j)$ so
\begin{equation} \lb{22.6}
\f{K_{n(j)} (x_0 + \f{z}{n\rho_n}, x_0 + \f{w}{n\rho_n})}{K_{n(j)}(x_0,x_0)} \to F(z,w)
\end{equation}
and $F$ is analytic in $w$ and anti-analytic in $z$.

\smallskip
(3) \ Note that by \eqref{22.2} and the Schwarz inequality \eqref{1.13a}, we have for $a,b\in\bbR$,
\begin{equation} \lb{22.7}
F(a,a) = 1 \qquad \abs{F(a,b)} \leq 1
\end{equation}
By compactness, if we show any such limiting $F$ is $\sin (\pi (\bar z-w))/(\bar z-w)$, we have
\eqref{22.4}. By analyticity, it suffices to prove this for $z=a$ real, and we will give details
when $z=0$, that is, we consider
\begin{equation} \lb{22.8}
\f{K_{n(j)}(x_0,x_0 + \f{z}{n\rho_n})}{K_{n(j)}(x_0,x_0)} \to f(z)
\end{equation}
\eqref{22.7} becomes
\[
f(0)=1 \qquad \abs{f(x)}\leq 1 \text{ for $x$ real}
\]

\smallskip
(4) \ By \eqref{1.12},
\begin{equation} \lb{22.9}
\int \abs{K_n(x_0, x_0 + a)}^2 w(a)\, da \leq K_n(x_0,x_0)
\end{equation}
which, by using the fact that $x_0$ is Lebesgue point, can be used to show
\begin{equation} \lb{22.10}
\int_{-\infty}^\infty \abs{f(x)}^2 \, dx \leq 1
\end{equation}

\smallskip
(5) \ By properties of $K_n$ (see Section~\ref{s6}) and Hurwitz's theorem, $f$ has zeros
$\{x_j\}_{j=-\infty,j\neq 0}^\infty$ only on $\bbR$, which we label by
\begin{equation} \lb{22.11}
\cdots < x_{-1} < 0 < x_1 < x_2 < \cdots
\end{equation}
and define $x_0=0$. By Theorem~\ref{T7.2}, using \eqref{22.2}, we have for any $j,k$ that
\begin{equation} \lb{22.12}
\abs{x_j-x_k} \geq \abs{j-k} -1
\end{equation}

\smallskip
(6) \ Given these facts, the theorem is reduced to

\begin{theorem} \lb{T22.2} Let $f$ be an entire function obeying
\begin{SL}
\item[{\rm{(1)}}]
\begin{equation} \lb{22.13}
f(0)=1 \qquad \abs{f(x)} \leq 1 \text{ for $x$ real}
\end{equation}

\item[{\rm{(2)}}]
\begin{equation} \lb{22.14}
\int_{-\infty}^\infty \abs{f(x)}^2\, dx \leq 1
\end{equation}

\item[{\rm{(3)}}] $f$ is real on $\bbR$, has only real zeros, and if they are labelled by \eqref{22.11},
then \eqref{22.12} holds.

\item[{\rm{(4)}}] For any $\veps$, there is a $C_\veps$ so
\begin{equation} \lb{22.16}
\abs{f(z)} \leq C_\veps e^{\veps\abs{z}^2}
\end{equation}
\end{SL}
Then
\begin{equation} \lb{22.17}
f(z) = \f{\sin \pi z}{\pi z}
\end{equation}
\end{theorem}

\begin{remarks} 1. There exist examples (\cite{ALSinprep}) $e^{-az^2 + bz} \sin\pi z/\pi z$ that
obey (1)--(3) and \eqref{22.16} for some but not all $\veps$.

\smallskip
2. We sketch the proof of this in case one has
\begin{equation} \lb{22.18}
\abs{f(z)} \leq Ce^{D\abs{z}}
\end{equation}
instead of \eqref{22.16}; see \cite{ALSinprep} for the general case.
\end{remarks}

\begin{lemma} \lb{L22.3} If {\rm{(1)--(3)}} hold and \eqref{22.18} holds, then for any $\veps >0$,
there is $D_\veps$ with
\begin{equation} \lb{22.19}
\abs{f(z)} \leq D_\veps e^{(\pi+\veps)\abs{\Ima z}}
\end{equation}
\end{lemma}

\smallskip
\noindent{\it Sketch}. By the Hadamard product formula \cite{Ahl},
\[
f(z)=e^{Dz} \prod_{j\neq 0} \biggl( 1-\f{z}{x_j}\biggr) e^{zx_j}
\]
where $D$ is real since $f$ is real on $\bbR$. Thus, for $y$ real,
\[
\abs{f(iy)}^2 = \prod_{j\neq 0} \biggl( 1+\f{y^2}{x_j^2}\biggr)
\]
By \eqref{22.12}, $\abs{x_j}\geq j-1$, so
\[
\abs{f(iy)}^2 \leq \biggl( 1 + \f{y^2}{x_1^2}\biggr)\biggl( 1+\f{y^2}{x_{-1}^2}\biggr)
\biggl[ \prod_{j=1}^\infty \biggl( 1 + \f{y^2}{j^2}\biggr)\biggr]^2
\]
which, given Euler's formula for $\sin \pi z/z$, implies \eqref{22.19} for $z=iy$. By a
Phragm\'en--Lindel\"of argument, \eqref{22.19} for $z$ real and for $z$ pure imaginary
and \eqref{22.18} implies \eqref{22.19} for all $z$.

\smallskip
Thus, Theorem~\ref{T22.2} (under hypothesis \eqref{22.18}) is implied by:

\begin{lemma}\lb{L22.4} If $f$ is an entire function that obeys \eqref{22.13}, \eqref{22.14},
and \eqref{22.19}, then \eqref{22.17} holds.
\end{lemma}

\begin{proof} Let $\hat f$ be the Fourier transform of $f$, that is,
\begin{equation} \lb{22.20}
\hat f(k) = (2\pi)^{-1/2} \int e^{-ikx} f(x)\, dx
\end{equation}
(in $L^2$ limit sense). By the Paley--Wiener Theorem \cite{RS2}, \eqref{22.19} implies
$\hat f$ is supported on $[-\pi,\pi]$. By \eqref{22.14},
\[
\Norm{\hat f}_{L^2} = \Norm{(2\pi)^{-1/2} \chi_{[-\pi,\pi]}}_{L^2} =1
\]
and, by \eqref{22.13} and support property of $\hat f$,
\[
\jap {\hat f, (2\pi)^{-1/2} \chi_{[-\pi,\pi]}} =1
\]
We thus have equality in the Schwarz inequality, so
\[
\hat f = (2\pi)^{-1/2} \chi_{[-\pi,\pi]}
\]
which implies \eqref{22.17}.
\end{proof}

This theorem has been applied in two ways:
\begin{SL}
\item[(a)] Lubinsky \cite{Lub-jdam} noted that one can recover Theorem~\ref{T20.3} from just
Totik's result Theorem~\ref{T17.3} without using the Lubinsky wiggle or Lubinsky's inequality.

\item[(b)] Avila--Last--Simon \cite{ALSinprep} have used this result to prove universality for
ergodic Jacobi matrices with a.c.\ spectrum where $\fre$ can be a positive measure Cantor set.
\end{SL}

\section{Zeros: The Freud--Levin--Lubinsky Argument} \lb{s23-new}

In the final section of his book \cite{FrB}, Freud proved bulk universality under fairly strong
hypotheses on measures on $[-1,1]$ and noticed that it implied a strong result on local
equal spacings of zeros. Without knowing of Freud's work, Simon, in a series of papers
(one joint with Last) \cite{S298,S299,S300,S309}, focused on this behavior, called it
clock spacing, and proved it in a variety of situations (not using universality or the CD
kernel). After Lubinsky's work on universality, Levin \cite{LL} rediscovered Freud's argument
and Levin--Lubinsky \cite{LL} used this to obtain clock behavior in a very general context.
Here is an abstract version of their result:

\begin{theorem}\lb{Tn23.1} Let $\mu$ be a measure of compact support on $\bbR$; let $x_0\in
\sigma(\mu)$ be such that for each $A$, for some $c_n$,
\begin{equation} \lb{n23.1}
\f{K_{n-1}(x_0 + \f{a}{nc_n}, x_0 + \f{b}{nc_n})}{K_{n-1} (x_0,x_0)} \to
\f{\sin(\pi(b-a))}{\pi(b-a)}
\end{equation}
uniformly for real $a,b$ with $\abs{a},\abs{b}\leq A$. Let $x_j^{(n)}(x_0)$ denote the zeros
of $p_n(x;d\mu)$ labelled so
\begin{equation} \lb{n23.2}
\cdots < x_{-1}^{(n)} (x_0) < x_0 \leq x_0^{(n)}(x_0) < x_1^{(n)}(x_0) < \cdots
\end{equation}
Then
\begin{SL}
\item[{\rm{(1)}}]
\begin{equation} \lb{n23.3}
\limsup\, nc_n (x_0^{(n)}-x_0) \leq 1
\end{equation}
\item[{\rm{(ii)}}] For any $J$, for large $n$, there are zeros $x_j^{(n)}$ for all $j\in\{-J, -J+1,
\dots, J-1, J\}$.
\item[{\rm{(iii)}}]
\begin{equation} \lb{n23.4}
\lim_{n\to\infty}\, (x_{j+1}^{(n)} - x_j^{(n)}) nc_n =1 \qquad \text{for each $j$}
\end{equation}
\end{SL}
\end{theorem}

\begin{remarks} 1. The meaning of $x_j^{(n)}$ has changed slightly from Section~\ref{s6}.

\smallskip
2. Only $nc_n$ enters, so the ``$n$'' could be suppressed; we include it because one expects
$c_n$ as defined to be bounded above and below. Indeed, in all known cases, $c_n\to\rho(x_0)$,
the derivative of the density of states. But see Remark~1 after Theorem~\ref{T22.1} for cases
where $c_n$ might not have a limit.

\smallskip
3. See \cite{LL2007} for the OPUC case.
\end{remarks}

\begin{proof} Let $\ti x_j^{(n)}(x_0)$ be the zeros of $p_n(x) p_{n-1}(x_0)-p_n(x_0) p_{n-1}(x)$
labelled as in \eqref{n23.2} (with $\ti x_0^{(n)}(x_0)=x_0$). By \eqref{n23.1}, we have $\ti x_{\pm 1}^{(n)}
(x_0) nc_n\to 1$ since $\sin(\pi a)/a$ is nonvanishing on $(-1,1)$ and vanishes at $\pm 1$. The same
argument shows $K_n (\ti x_{\pm 1}^{(n)}, \ti x_{\pm 1}^{(n)} + b/nc_n)$ is nonvanishing for
$\abs{b}<\f12$, and so there is at most one zero near $\ti x_{\pm 1}^{(n)}$ on $1/nc_n$ scale. It
follows by repeating this argument that
\begin{equation} \lb{n23.5}
nc_n \ti x_j^{(n)}\to j
\end{equation}
for all $j$.

Since we have (see Section~\ref{s6}) that
\[
x_0\leq x_0^{(n)}(x_0) \leq \ti x_1^{(n)} (x_0)
\]
by interlacing, which implies (i) and similar interlacing gives (ii). Finally, \eqref{n23.4} follows from
the same argument that led to \eqref{n23.5}.
\end{proof}

\section{Adding Point Masses} \lb{s23}

We end with a final result involving CD kernels---a formula of Geronimus \cite[formula (3.30)]{GBk}.
While he states it only for OPUC, his proof works for any measure on $\bbC$ with finite moments. Let
$\mu$ be such a measure, let $z_0\in\bbC$, and let
\begin{equation} \lb{23.1}
\nu=\mu+\lambda \delta_{z_0}
\end{equation}
for $\lambda$ real and bigger than or equal to $-\mu(\{z_0\})$.

Since $X_n(z;d\nu)$ and $X_n (z;d\mu)$ are both monic, their difference is a polynomial of degree $n-1$,
so
\begin{equation} \lb{23.2}
X_n(x;d\nu)=X_n(z;d\mu) + \sum_{j=0}^{n-1} c_j x_j (z;d\mu)
\end{equation}
where
\begin{align}
c_j &= \int \ol{x_j (z;d\mu)}\, [X_n(z;d\nu) - X_n (z;d\mu)]\, d\mu \lb{23.3} \\
&= \int \ol{x_j(z;d\mu)}\, X_n (z;d\nu) [d\nu-\lambda\delta_{z_0}] \lb{23.4} \\
&= -\lambda\, \ol{x_j(z_0;d\mu)}\, X_n(z_0;d\nu) \lb{23.5}
\end{align}
where \eqref{23.4} follows from $x_j (\dott,d\mu)\perp X_n (\dott,d\mu)$ in $L^2 (d\mu)$ and
\eqref{23.5} from $x_j (\dott,d\mu)\perp X_n (\dott,d\nu)$ in $L^2 (d\nu)$. Thus,
\begin{equation} \lb{23.5x}
X_n(z;d\nu) = X_n (z;d\mu) -\lambda X_n(z_0;d\nu) K_{n-1} (z_0,z;d\mu)
\end{equation}
Set $z=z_0$ and solve for $X_n(z_0;d\nu)$ to get:

\begin{theorem}[Geronimus \cite{GBk}]\lb{T23.1} Let $\mu,\nu$ be related by \eqref{23.1}. Then
\begin{equation} \lb{23.6}
X_n(z;d\nu) = X_n(z;d\mu) - \f{\lambda X_n(z_0;d\mu) K_{n-1}(z_0,z;d\mu)}
{1+\lambda K_{n-1} (z_0,z_0;d\mu)}
\end{equation}
\end{theorem}

This formula was rediscovered by Nevai \cite{Nev79} for OPRL, by Cachafeiro--Marcell\'an
\cite{CM88a,CM88b}, Simon \cite{OPUC2} (in a weak form), and Wong \cite{Wong,Wong-opsfa} for
OPUC. For general measures on $\bbC$, the formula is from Cachafeiro--Marcell\'an \cite{CM89,CM91}.
In particular, in the context of OPUC, Wong \cite{Wong-opsfa} noted that one can use the CD formula
to obtain:

\begin{theorem}[Wong \cite{Wong,Wong-opsfa}]\lb{T23.2} Let $d\mu$ be a probability measure on
$\partial\bbD$ and let $d\ti\nu$ be given by
\begin{equation} \lb{23.7}
d\ti\nu = \f{d\mu+\lambda\delta_{z_0}}{1+\lambda}
\end{equation}
for $z_0\in\partial\bbD$ and $\lambda\geq -\mu(\{z_0\})$. Then
\begin{equation} \lb{23.8}
\alpha_n (d\ti\nu) = \alpha_n (d\mu) + \f{(1-\abs{\alpha_n (d\mu)}^2)^{1/2}}
{\lambda^{-1} + K_n (z_0,z_0;d\mu)}\, \ol{\varphi_{n+1}(z_0)}\, \varphi_n^*(z_0)
\end{equation}
\end{theorem}

\begin{proof} Let
\begin{equation} \lb{23.9}
Q_n = \lambda^{-1} + K_n(z_0,z;d\mu)
\end{equation}
Since
\begin{equation} \lb{23.10}
\Phi_{n+1}(z;d\ti\nu) = \Phi_{n+1}(z;d\nu)
\end{equation}
and
\begin{equation} \lb{23.11}
\alpha_n (d\ti\nu) = \ol{-\Phi_{n+1} (0;d\ti\nu)}
\end{equation}
\eqref{23.6} becomes
\begin{equation} \lb{23.12}
\alpha_n (d\ti\nu) -\alpha_n (d\mu) = Q_n^{-1}\,\,\, \ol{\Phi_{n+1}(z_0)}\,
\ol{K_n(z_0,0;d\mu)}
\end{equation}
By the CD formula in the form \eqref{3.18},
\begin{align}
\ol{K_n(z_0,0)} &= \varphi_n^*(z_0)\, \ol{\varphi_n^*(0)} \lb{23.13} \\
&= \f{\varphi_n^*(z_0)}{\Norm{\Phi_n}} \lb{23.14}
\end{align}
since $\Phi_n^*(0)=1$ and $\Norm{\Phi_n^*} = \Norm{\Phi_n}$. \eqref{23.8} then follows
from $\Norm{\Phi_{n+1}}/\Norm{\Phi_n} = (1-\abs{\alpha_n}^2)^{1/2}$.
\end{proof}

To see a typical application:

\begin{corollary}\lb{C23.3} Let $z_0$ be an isolated pure point of a measure $d\mu$ on
$\partial\bbD$. Let $d\ti\nu$ be given by \eqref{23.7} where $\lambda > -\mu(\{z_0\})$
{\rm{(}}so $z_0$ is also a pure point of $d\ti\nu${\rm{)}}. Then for some $D,C>0$,
\begin{equation} \lb{23.15}
\abs{\alpha_n (d\ti\nu)-\alpha_n (d\mu)} \leq De^{-Cn}
\end{equation}
\end{corollary}

\begin{proof} By Theorem~10.14.2 of \cite{OPUC2},
\begin{equation} \lb{23.16}
\abs{\varphi_n (z_0;d\mu)}\leq D_1 e^{-\f12 Cn}
\end{equation}
This plus \eqref{23.8} implies \eqref{23.15}.
\end{proof}

This is not only true for OPUC but also for OPRL:

\begin{corollary}\lb{C23.4} Let $z_0$ be an isolated pure point of a measure of compact support
$d\mu$ on $\bbR$. Let $d\ti\nu$ be given by \eqref{23.7} where $\lambda >-\mu(\{z_0\})$ so $z$
is also a pure point of $d\ti\nu$. Then for some $D,C >0$,
\begin{alignat}{2}
&\text{\rm{(i)}} \qquad && \biggl| \f{\kappa_n(d\nu)}{\kappa_n(d\mu)} - (1+\lambda)^{1/2}\biggr|
\leq De^{-Cn} \lb{23.17} \\
&\text{\rm{(ii)}} \qquad && \Norm{p_n(\dott,d\nu) - (1+\lambda)^{1/2} p_n (\dott,d\mu)}_{L^2(d\nu)}
\leq De^{-Cn} \lb{23.18} \\
&\text{\rm{(iii)}} \qquad && \abs{a_n(d\nu)-a_n(d\mu)} \leq De^{-Cn} \lb{23.19} \\
&{} \qquad && \abs{b_n (d\nu) - b_n (d\mu)} \leq De^{-Cn} \lb{23.20}
\end{alignat}
\end{corollary}

\begin{proof}[Sketch] Isolated points in the spectrum of Jacobi matrices obey
\begin{equation} \lb{23.21}
\abs{p_n(z_0)} \leq D_1 e^{-C_1 n}
\end{equation}
for suitable $C_1,D_1$ (see \cite{Agmon,ComTh}).

\eqref{23.6} can be rewritten for OPRL
\begin{equation} \lb{23.22}
\kappa_n(d\mu) P_n(x;d\nu) = p_n (x;d\mu) -\lambda p_n (x_0;d\mu)\,
\f{K_{n-1} (x_0,x;d\mu)}{1+\lambda K_n(x_0,x_0;d\mu)}
\end{equation}
Since
\[
\Norm{K_{n-1}(x_0,x;d\mu)}^2_{L^2(d\mu)} = K_{n-1} (x_0,x_0;d\mu)
\]
and
\[
\int \abs{K_{n-1} (x_0,x_0;d\mu)}^2\, d\delta_{x_0} = K_{n-1} (x_0,x_0;d\mu)^2
\]
and $K_{n-1}(x_0,x_0)$ is bounded (by  \eqref{23.21}), we see that $\Norm{K_{n-1}(x_0;\dott;d\mu)}_{L^2(d\nu)}$
is bounded. Thus, by \eqref{23.21} and \eqref{23.22},
\[
\kappa_n(d\mu) \kappa_n(d\nu)^{-1} = (1+\lambda)^{-1/2} + O(e^{-C_1n})
\]
which leads to \eqref{23.17}.

This in turn leads to (ii), and that to (iii) via \eqref{23.21}, and, for example,
\begin{align}
a_n(d\mu) &= \int xp_n (x;d\mu) p_{n-1} (x;d\mu)\, d\mu \lb{23.23} \\
a_n (d\nu) &= \int xp_n (x;d\nu) p_{n-1} (x;d\nu)\, d\nu \lb{23.24}
\end{align}
\end{proof}

This shows what happens if the weight of an isolated eigenvalue changes. What happens if an isolated
eigenvalue is totally removed is much more subtle---sometimes it is exponentially small, sometimes not.
This is studied by Wong \cite{Wong-prep}.

\bigskip

\end{document}